\documentclass[a4paper,leqno,10pt]{article}
\usepackage{latexsym,bm}
\usepackage{amsmath,amsfonts,amssymb,amsthm,mathrsfs}
\usepackage{graphics}

\DeclareMathOperator{\End}{End} \DeclareMathOperator{\ch}{char}
 \DeclareMathOperator{\Ker}{Ker}
 \DeclareMathOperator{\sign}{sign}
\DeclareMathOperator{\id}{id} 
 
\DeclareMathOperator{\Bd}{Bd} \DeclareMathOperator{\Ann}{Ann}
\DeclareMathOperator{\Std}{Std} \DeclareMathOperator{\std}{Std}
 
\DeclareMathOperator{\cont}{cont} \DeclareMathOperator\Shape{Shape}
\def\Add{\mathscr A}
\def\Rem{\mathscr R}

\numberwithin{equation}{section}

\numberwithin{equation}{section}
\newtheorem{theorem}[equation]{Theorem}
\newtheorem{lemma}[equation]{Lemma}

\newtheorem{corollary}[equation]{Corollary}

\newtheorem{definition}[equation]{Definition}

\newtheorem{remark}[equation]{Remark}

\let\<=\langle
\let\>=\rangle
\def\({\big(}
\def\){\big)}
\def\lam{\lambda}
\def\s{{\mathfrak s}}
\def\t{{\mathfrak t}}
\def\tlam{\t^{\lambda}}

\def\u{{\mathfrak u}}
\def\v{{\mathfrak v}}

\def\bS{{\mathtt S}}

\newcommand{\pp}{\mathcal{P}}
\newcommand{\C}{\mathbb{C}}
\newcommand{\N}{\mathbb N}
\newcommand{\Q}{\mathbb Q}
\newcommand{\Z}{\mathbb Z}
\newcommand{\Sym}{\mathfrak S}
\def\bb{\mathfrak{B}}
{\setlength{\textwidth}{125mm}} {\setlength{\textheight}{185mm}}

\title{On a theorem of Lehrer and Zhang}

\author{Jun Hu$^{\ast}$ and Zhankui Xiao$^{\dagger}$\\[5pt]
$^\ast$School of Mathematics, Beijing Institute of Technology\\
Beijing, 100081, P.R. China\\
\&\\
School of Mathematics and Statistics\\
University of Sydney,
NSW, 2006, Australia\\ [1pt]
Email: junhu303@yahoo.com.cn\\[3pt]
$^\dagger$School of Mathematical Sciences, Huaqiao University\\[1pt]
Quanzhou, Fujian, 362021, P. R. China\\ Email: zhkxiao@gmail.com}

\begin{document}

\maketitle
\begin{abstract} Let $K$ be an arbitrary field of characteristic
not equal to $2$. Let $m, n\in\N$ and $V$ an $m$ dimensional
orthogonal space over $K$. There is a right action of the
Brauer algebra $\bb_n(m)$ on the $n$-tensor space
$V^{\otimes n}$ which centralizes the left action of the orthogonal
group $O(V)$. Recently G.I.~Lehrer and R.B.~Zhang defined certain quasi-idempotents $E_i$ in $\bb_n(m)$ (see (\ref{keydfn})) and proved that the annihilator of $V^{\otimes n}$ in $\bb_n(m)$ is always equal to the
two-sided ideal generated by $E_{[(m+1)/2]}$ if $\ch K=0$ or $\ch K>2(m+1)$. In this paper we extend this theorem to arbitrary field $K$ with $\ch K\neq 2$ as conjectured by Lehrer and Zhang. As a byproduct, we discover a combinatorial identity which relates to the dimensions of Specht modules over the symmetric groups of different sizes and a new integral basis for the annihilator of $V^{\otimes m+1}$ in $\bb_{m+1}(m)$.
\end{abstract}

\section{Introduction}\label{xxsec1}

Let $\N$ be the set of non-negative integers. Let $x$ be an indeterminate over $\Z$ and $0<n\in\N$. The Brauer
algebra $\bb_n(x)$ over $\Z[x]$ was introduced by Richard
Brauer (see \cite{Br}) when he studied how the $n$-tensor space
$V^{\otimes{n}}$ decomposes into irreducible modules over the
orthogonal group $O(V)$ or the symplectic group $Sp(V)$, where $V$
is an orthogonal vector space or a symplectic vector space. It was
defined as the free $\Z[x]$-module on the basis of the set $\Bd_n$
of all Brauer $n$-diagrams, graphs on $2n$ vertices, and $n$ edges with
the property that every vertex is incident to precisely one edge.
The multiplication of two Brauer n-diagrams is defined using natural
concatenation of diagrams. Precisely, we compose two diagrams $D_1,
D_2$ by identifying the bottom row of vertices in $D_1$ with the top
row of vertices in $D_2$. The result is a graph, with a certain
number, $n(D_1,D_2)$, of interior loops. After removing the interior
loops and the identified vertices, retaining the edges and remaining
vertices, we obtain a new Brauer $n$-diagram $D_1\circ D_2$, the
composite diagram. Then we define $D_1\cdot
D_2=x^{n(D_1,D_2)}D_1\circ D_2$. For example, let $d$ be the
following Brauer $5$-diagram.
\medskip
\begin{center}
\scalebox{0.3}[0.3]{\includegraphics{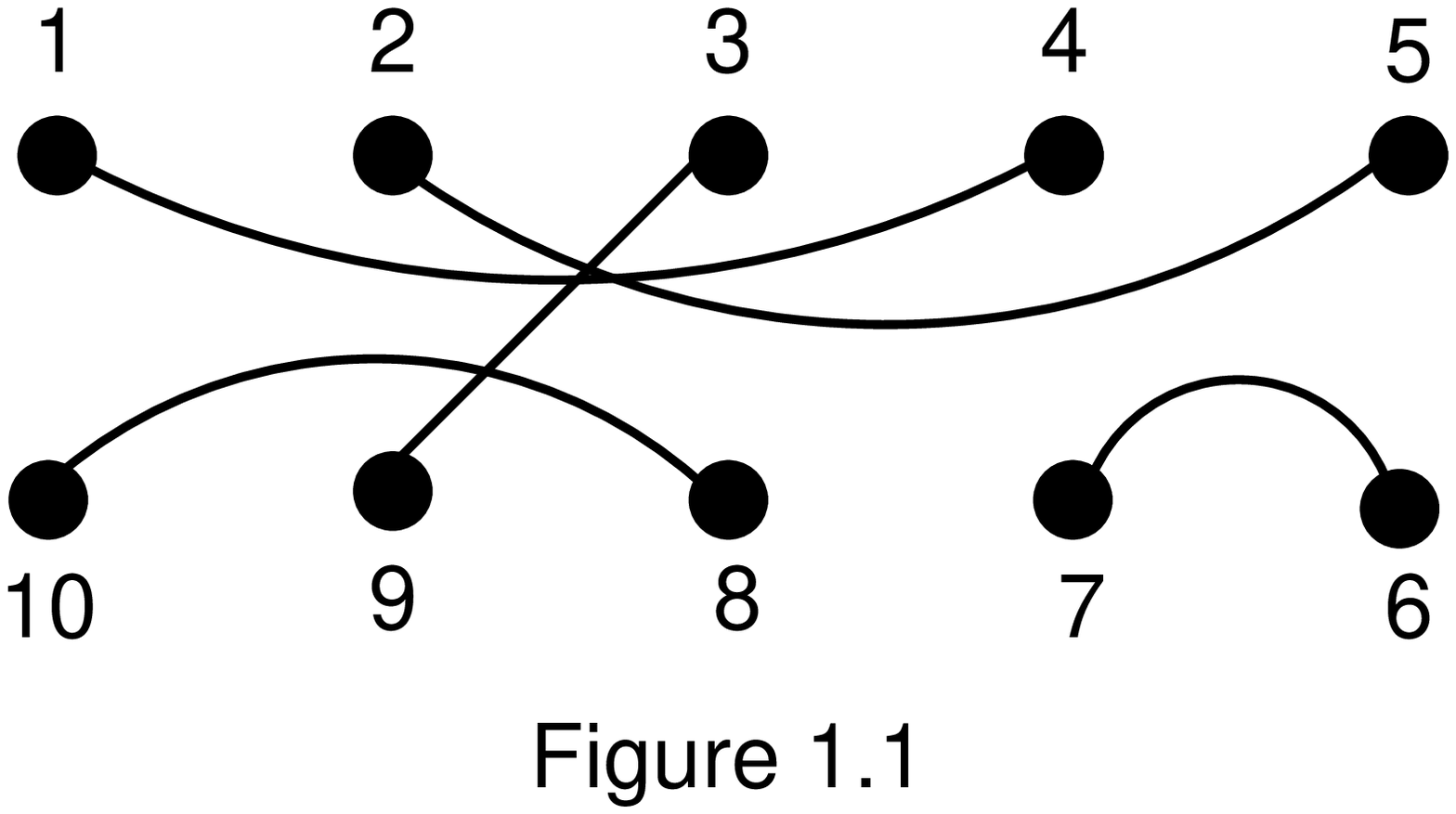}}
\end{center}
\medskip
Let $d'$ be the following Brauer $5$-diagram.
\medskip
\begin{center}
\scalebox{0.3}[0.3]{\includegraphics{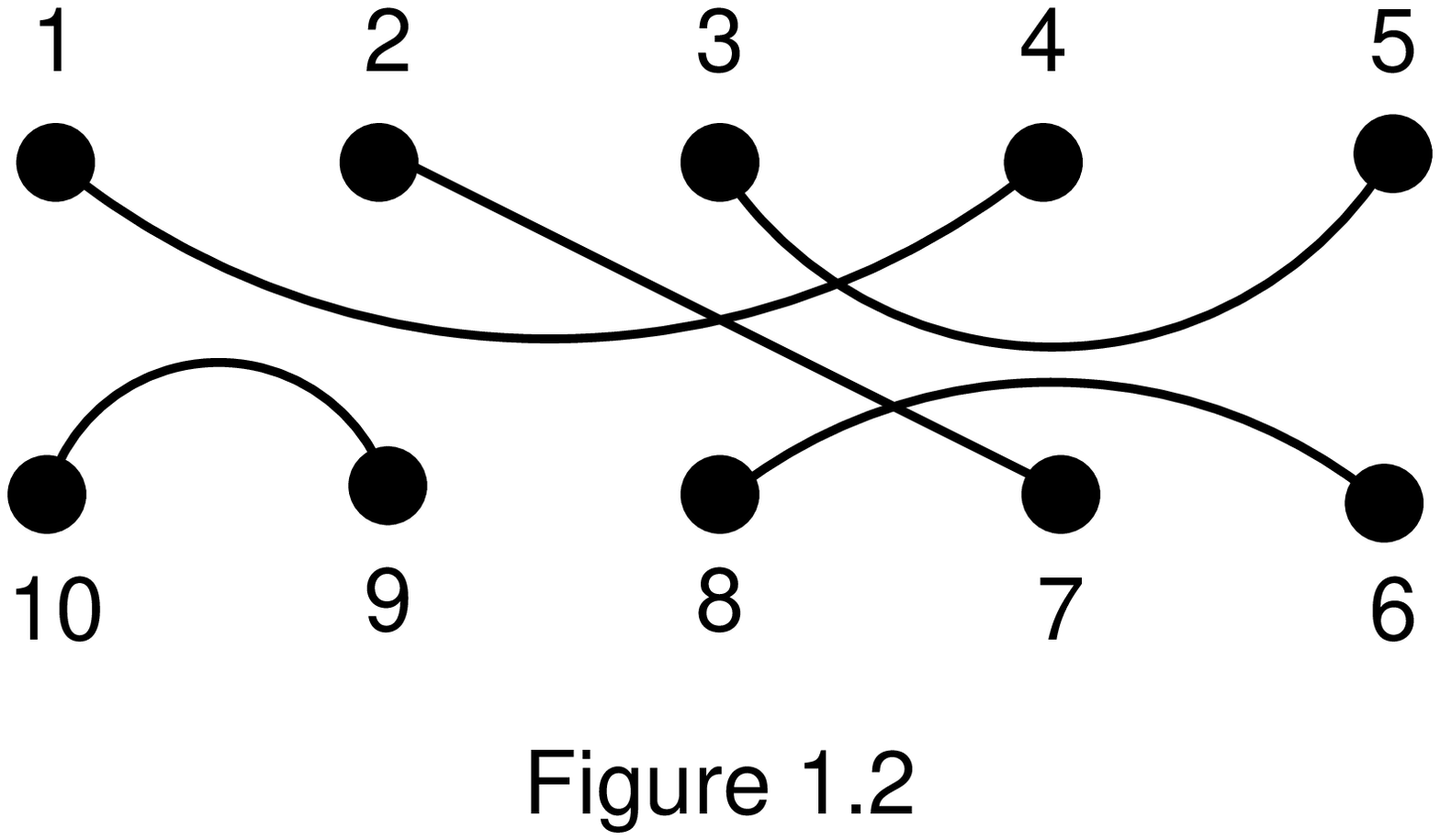}}
\end{center}
\medskip
Then $dd'$ is equal to
\medskip
\begin{center}
\scalebox{0.53}[0.53]{\includegraphics{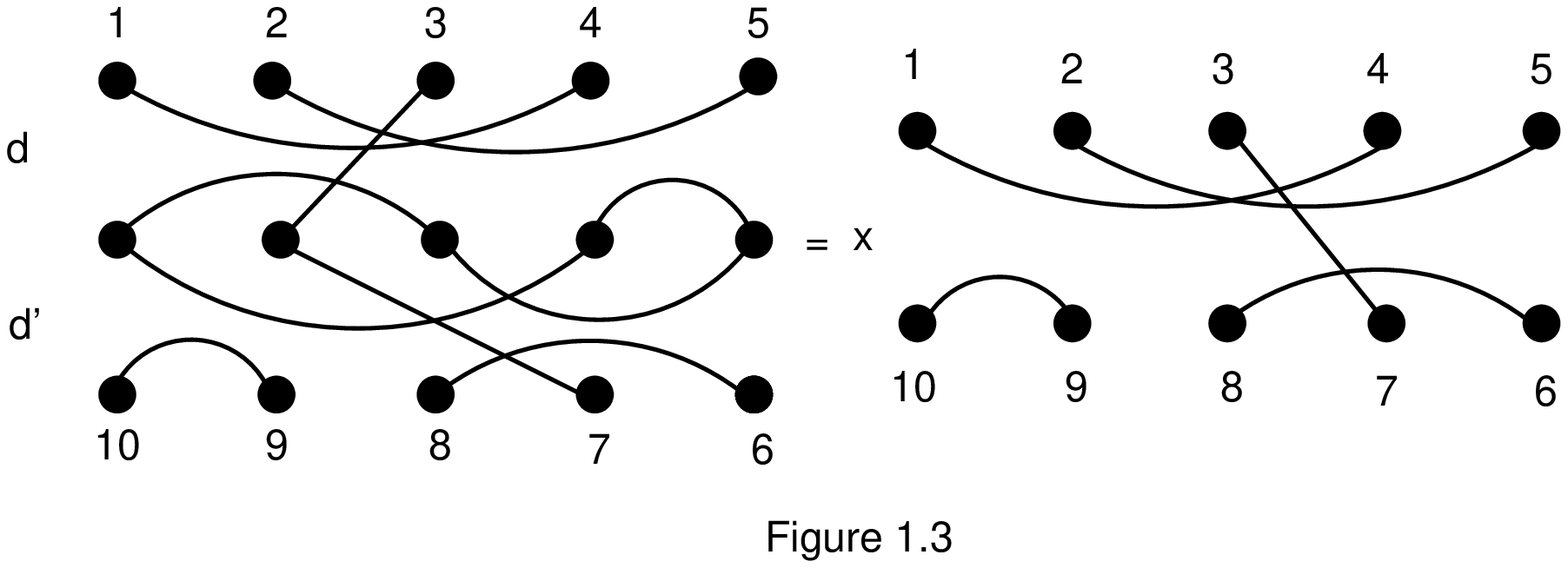}}
\end{center}
\medskip
In general, the multiplication of two elements in $\bb_n(x)$ is
given by the linear extension of a product defined on diagrams.
For each integer $i$ with $1\leq i\leq 2n$, we define $i^{-}:=2n+1-i$.
The Brauer algebra $\bb_n(x)$ is a free $\Z[x]$-module with rank $(2n-1)\cdot
(2n-3)\cdots 3\cdot 1$. For any $\Z[x]$-algebra $R$ with $x$
specialized to $\delta\in R$, we define $\bb_n(\delta)_{R}:=R\otimes_{\Z[x]}\bb_n(x)$.

Now let $K$ be an arbitrary field of characteristic not equal to $2$.
Let $m, n$ be two positive integers and $V$ an $m$ dimensional orthogonal space over
$K$. Let $\bb_n(m)$ be the specialized Brauer algebra
with parameter $m\cdot 1_{K}$. There is a right action of $\bb_n(m)$ on the
$n$-tensor space $V^{\otimes n}$ which commutes with the left action
of the orthogonal group $O(V)$. If $K=\mathbb{C}$, then by a
well-known result of Brauer \cite{Br}, the canonical
homomorphism $\varphi: \bb_n(m)\rightarrow\End_{O(V)}(V^{\otimes
n})$ is surjective. In general, as long as $K$ is an infinite field
of characteristic not equal to $2$, the surjection still holds and
we actually have a characteristic-free version of the Schur--Weyl
duality between $\bb_n(m)$ and $KO(V)$ on $V^{\otimes n}$. For the
proof as well as the symplectic version of these results, we refer the readers to
\cite{DP}, \cite{DDH} and \cite{DH}.\smallskip

The above Schur--Weyl duality is closely related to the second
fundamental theorem in invariant theory for $O(V)$. By \cite{DWH}
and \cite{RS}, $\bb_n(m)$ is semisimple if and only if $m\geq n-1$.
From the representation theoretic point of view, it is desirable to
describe the radical of $\bb_n(m)$ in the non-semisimple case. By
\cite{Ga2}, the kernel of $\varphi$ is closely related to the
radical of $\bb_n(m)$. Therefore, it is important to understand
the kernel of $\varphi$. Note that $\varphi$ is not injective if and
only if $n\geq m+1$. In \cite[Theorem 4.8]{Ga2}, using the invariant
theory for $O(V)$, Gavarini showed that the kernel of $\varphi$ is
spanned by some diagrammatic minors of order $m+1$ (which are
certain alternating sum of some Brauer $n$-diagrams). Note that,
however, those diagrammatic minors are not necessarily $K$-linearly
independent. In \cite[Theorem 1.4, Theorem 6.9]{DH}, an integral
basis for the kernel of $\varphi$ was obtained. The Brauer algebra
$\bb_n(m)$ can be endowed with a right $\Sym_{2n}$-module structure
in a way such that $\Ker\varphi$ is an $\Sym_{2n}$-submodule
(cf. \cite{DH} and \cite{Ga1}). So far, to the best of our
knowledge, it remains an open question on whether or not there
exists a characteristic-free basis for $\Ker\varphi$ which consists
of some diagrammatic minors of order $m+1$.\smallskip

In \cite[Corollary 5.9]{HX1}, we proved in the symplectic case that
$\Ker\varphi$ is always generated by one specific
diagrammatic Pfaffian of order $2m+2$. In the quantized type $C$
case, we proved (in \cite[Proposition 5.6]{HX1}) a similar statement
under the assumption that the quantum parameter $q$ is generic.
Recently, G.I. Lehrer and R.B. Zhang have studied extensively the orthogonal case in \cite{LZ2} by connecting it with  the second fundamental theorem of invariant theory for the orthogonal group.
For each Brauer $n$-diagram $D\in\Bd_n$, the vertices of $D$ are
arranged in two rows: the top and bottom rows.
The vertices in top row are labeled by the indices $1,2,\cdots,n$ from left to right; while
the vertices in bottom row are labeled by the indices
$1^-,\cdots,n^-$ from left to right. The following key definitions are due to them.

\begin{definition} \label{keydfn} {\rm (\cite[Definition 4.2]{LZ2})} Let $a,b\in\N$ such that $1\leq a+b\leq
n$. Let $\Bd(a,b)$ be the set of all Brauer $n$-diagrams $D$ such
that: \begin{enumerate}
\item[(1)] for each integer $s$ with $a+b+1\leq s\leq n$, $D$ connects the vertex
labeled by $s$ with the vertex labeled by $s^-$; and
\item[(2)] for each integer $s$ with $s\in\{1,2,\cdots,a,(a+1)^-,(a+2)^-,\cdots,(a+b)^-\}$, $D$ connects the vertex
labeled by $s$ with the vertex labeled by $t$ for
$t\in\{1^-,2^-,\cdots,a^-,a+1,a+2,\cdots,a+b\}$.
\end{enumerate}
We define $$ E_{a,b}:=\sum_{D\in \Bd(a,b)}\sign(D)D,\quad\,
E_i:=E_{i,m+1-i},\,\,\forall\,\,0\leq i\leq m+1.
$$
\end{definition}
We refer the reader to Definition \ref{dfn000} for the definition of $\sign(D)$.\footnote{At a first look, the definition of $E_i$ which we give here seems to be different with \cite[Definition 4.2]{LZ2}. But they are indeed the same. The equivalence between the two definitions follows
from a simple counting by the proof given in the paragraph directly below \cite[(4.5)]{LZ2}.}
Lehrer and Zhang have proved a number of important properties about those $E_i$. In particular, they have proved the following theorem in \cite{LZ2}.

\begin{theorem} {\rm (\cite[Proposition 4.4, Theorems 4.3, 9.4]{LZ2})} \label{LZ} Assume that $m<n$. Then for each integer $i$ with $0\leq i\leq [(m+1)/2]$, $E_i\in\Ker\varphi$. Furthermore, if $\ch K=0$ or $\ch K>2(m+1)$ then $\Ker\varphi$ is the two-sided ideal of $\bb_n(m)$ generated by $E_{[(m+1)/2]}$.
\end{theorem}

Furthermore, Lehrer and Zhang have conjectured in \cite[Remark 9.5]{LZ2} that the second statement of the above theorem is true for arbitrary field $K$ with $\ch K\neq 2$. The main result in this paper is a proof of this conjecture. In other words, we extend Lehrer and Zhang's theorem to arbitrary field $K$ with $\ch K\neq 2$. That is,

\begin{theorem} \label{conj} Let $K$ be an arbitrary field of characteristic other than two. Then $\Ker\varphi$ is always equal to the two-sided ideal generated by $E_{[(m+1)/2]}$.
\end{theorem}

As a byproduct, we discover (in Theorem \ref{mainthm2} and Corollary \ref{maincor}) a combinatorial identity which connects to the dimensions of some Specht modules over the symmetric group $\Sym_{2m+2}$ and the symmetric group $\Sym_{m+1}$. We get (in Corollary \ref{maincor2}) a new integral basis for the annihilator of $V^{\otimes m+1}$ in $\bb_{m+1}(m)$. The content is organized as follows. In Section 2 we recall some basic knowledge
about the structure and representation theory of the Brauer algebras
as well as some related combinatorics which are needed later. In Section 3 we prove that the annihilator of $V^{\otimes n}$ in $\bb_n(m)$ is equal to the two-sided ideal generated by $E_0,E_1,\dots,E_{[(m+1)/2]}$.
The proof makes essential use of the integral basis of $\Ker\varphi$
obtained in \cite{DH}. In Section 4 we prove our main result Theorem \ref{conj}.
The proof will proceed in three steps. The main strategy to
prove Theorem \ref{conj} is to transform it into a statement about
identification between certain two-sided ideals in the symmetric group
algebra $K\Sym_n$. For the latter, we make use of the Young
seminormal basis and the Murphy basis theory of the symmetric group
algebras as well as the first main result obtained in Section 3.

\section{The Brauer algebra}

The Brauer algebra $\bb_n(x)$ can be alternatively defined as the
unital associative $\Z[x]$-algebra with generators
$s_1,\cdots,s_{n-1},e_1,$ $\cdots,e_{n-1}$ and relations (see
\cite{En}):
$$\begin{matrix}s_i^2=1,\,\,e_i^2=xe_i,\,\,e_is_i=e_i=s_ie_i,
\quad\forall\,1\leq i\leq n-1,\\
s_is_j=s_js_i,\,\,s_ie_j=e_js_i,\,\,e_ie_j=e_je_i,\quad\forall\,1\leq
i<j-1\leq n-2,\\ s_is_{i+1}s_i=s_{i+1}s_is_{i+1},\,\,
e_ie_{i+1}e_i=e_i,\,\, e_{i+1}e_ie_{i+1}=e_{i+1},\,\,\forall\,1\leq
i\leq n-2,\\
s_ie_{i+1}e_i=s_{i+1}e_i,\,\,e_{i+1}e_is_{i+1}=e_{i+1}s_i,\quad\forall\,1\leq
i\leq n-2.\end{matrix}
$$
 Note that the
subalgebra of $\bb_n(x)$ generated by $s_1,\cdots,s_{n-1}$ is
isomorphic to the symmetric group algebra $\Z[x]\Sym_n$ of $\Sym_n$
over $\Z[x]$ and $s_1,\cdots,s_{n-1}$ are the standard Coxeter
generators. Let $\ell:{\Sym_n}\rightarrow\N$ be the length function
on $\Sym_n$ so that $\ell(w)=k$ if $k$ is minimal such that
$w=s_{i_1}\dots s_{i_k}$, for some $s_{i_j}$ with $1\le i_j<n$.

For each integer $1\leq
j<n$, the generator $s_j$ corresponds to the Brauer $n$-diagram with
edges connecting the vertices $j$ (respectively, $j+1$) on the top row
with $(j+1)^-$ (respectively, $j^-$) on the bottom row, and all other edges
are vertical, connecting the vertices $k$ and $k^-$ on the top and bottom rows for
all $k\neq i,i+1$; the generator $e_j$ corresponds to the Brauer
$n$-diagram with horizontal edges connecting the vertices $j, j+1$ (resp., $j^-, (j+1)^-$) on the
top rows (resp., bottom rows), and all other edges are vertical, connecting the
vertices $k$ and $k^-$ on the top and bottom rows for all $k\neq j,j+1$.

Let $R$ be a commutative integral domain which is an $\Z[x]$-algebra
such that $x$ is specialized to $\delta\in R$. Then both the
symmetric group algebra $R\Sym_n$ and the Brauer algebra
$\bb_n(\delta)_R$ are cellular algebras over $R$ (see \cite{Murphy:basis} and \cite{GL}). To recall their
cellular structures we need some combinatorics. A composition of $n$
is a sequence of nonnegative integer $\lam=(\lam_1,\lam_2,\cdots)$
with $\sum_{i\geq 1}\lam_i=n$. A composition
$\lam=(\lam_1,\lam_2,\cdots)$ of $n$ is said to be a partition if
$\lam_1\geq\lam_2\geq\cdots$. In this case, we write $\lam\vdash n$
and $|\lam|=n$. We use $\mathcal{P}_n$ to denote the set of all the
partitions of $n$. For any composition $\lam$ of $n$, the conjugate
of $\lam$ is defined to be a partition
$\lam'=(\lam'_1,\lam'_2,\cdots)$, where $\lam'_j:=\#\{i|\lam_i\geq
j\}$ for each $j\geq 1$. We use $\Sym_{\lam}$ to denote the standard
Young subgroup of $\Sym_n$ corresponding to $\lam$. That is $$
\Sym_{\lam}:=\Sym_{\{1,2,\cdots,\lam_1\}}\times\Sym_{\{\lam_1+1,\lam_1+2,\cdots,\lam_1+\lam_2\}}\times\cdots.
$$

Let $\lambda$ be a composition of $n$. The Young diagram of $\lam$
is defined to be the set
$$ [\lam]:=\bigl\{(i,j)\bigm|1\leq j\leq \lam_i\bigr\}.
$$
The elements of $[\lam]$ are called nodes of $\lam$. A
$\lam$-tableau is a bijection $\t:
[\lam]\rightarrow\{1,2,\cdots,n\}$. The symmetric group $\Sym_n$
acts on the set of $\lam$-tableaux from the right hand side by
letter permutations. If $\lam$ is a partition, then the conjugate of
$\t$ is define to be the $\lam'$-tableau $\t'$ such that
$\t'(i,j):=\t(j,i)$ for any $(i,j)\in[\lam']$. The $\lam$-tableau
$\t$ is row standard if $ \t(i,j)\leq\t(i,k)$ whenever $j\leq k$.
$\t$ is standard if both $\t$ and $\t'$ are row-standard. Let
$\std(\lam)$ be the set of standard $\lam$-tableaux. We denote by
$\t^{\lam}$ (respectively, $\t_{\lam}$) the standard $\lam$-tableau
in which the numbers $1,2,\cdots,n$ appear in order along successive
rows (respectively, columns). If $\t$ is a $\lam$-tableau then let
$d(\t)\in\Sym_n$ such that $\t^{\lam}d(\t)=\t$ and we shall write
$\Shape(\t)=\lam$. Note that $\Sym_{\lam}$ is the row stabilizer of
$\tlam$. We use $\mathcal{D}_{\lam}$ to denote the set of
distinguished right coset representatives of $\Sym_{\lam}$ in
$\Sym_n$. Then for any $d\in\mathcal{D}_{\lam}$ and
$w\in\Sym_{\lam}$ we have that $\ell(wd)=\ell(w)+\ell(d)$. Let
$w_{\lam}\in\Sym_n$ such that $\t^{\lam}w_{\lam}=\t_{\lam}$. Then
$w_{\lam}\in\mathcal{D}_{\lam}$.

We define
$$ X_{\lam}:=\sum_{w\in\Sym_{\lam}}w,\quad
Y_{\lam}:=\sum_{w\in\Sym_{\lam}}(-1)^{\ell(w)}w.
$$
Let $\tau$ be the $R$-algebra automorphism of $R\Sym_n$ which is
defined on generators by $\tau(s_i)=-s_i$ for any $1\leq i<n$. It is
clear that $\tau^2=\id$ and $\tau(Y_{\lam})=X_{\lam}$.

Let $\lam\vdash n$. For any $\s,\t\in\std(\lam)$, we define
$X_{\s\t}:=d(\s)^{-1}X_{\lam}d(\t)$. By a well-known result of
Murphy \cite{Murphy:basis}, the set $\{X_{\s\t}|\lam\vdash n,
\s,\t\in\std(\lam)\}$ is a basis of $R\Sym_n$. We call it the {\it
Murphy basis} of $R\Sym_n$. It is a cellular basis of $R\Sym_n$ in
the sense of \cite{GL}. Note also that the set $\{Y_{\s\t}:=d(\s)^{-1}Y_{\lam}d(\t)|\lam\vdash n,
\s,\t\in\std(\lam)\}$ is a cellular basis of $R\Sym_n$ too. We call it the {\it Y
Murphy basis} of $R\Sym_n$. For both cellular bases the cell modules (i.e., Specht modules) of $R\Sym_n$ are
labeled by the partitions in $\mathcal{P}_n$.

For any $\lambda,\mu\in\mathcal{P}_n$, we write $\lambda\unrhd\mu$
if $\sum_{j=1}^i\lambda_j\geq \sum_{j=1}^i\mu_j$ for any $i\geq 1$.
If $\lambda\unrhd\mu$ and $\lambda\neq\mu$, then we write
$\lambda\rhd\mu$. We use $(R\Sym_n)^{\unrhd\lam}$ (respectively,
$(R\Sym_n)^{\rhd\lam}$) to denote the free $R$-submodule of
$R\Sym_n$ spanned by the Murphy basis elements of the form
$X_{\u\v}$ with $\u,\v\in\std(\mu)$ and $\mu\unrhd\lam$
(respectively, $\mu\rhd\lam$). Then both $(R\Sym_n)^{\unrhd\lam}$
and $(R\Sym_n)^{\rhd\lam}$ are two-sided ideals of $R\Sym_n$.
\smallskip

We now recall the cellular structure of the Brauer algebra
$\bb_n(m)$. Let $f$ be an integer with $0\leq f\leq [n/2]$, where
$[n/2]$ is the largest non-negative integer not bigger than $n/2$.
We define $$
\mathcal{D}_{f}:=\Biggl\{d\in\Sym_n\Biggm|\begin{matrix}\text{$(2j-1)d<(2j)d$
for $1\leq j\leq f$}\\
\text{$(1)d<(3)d<\cdots<(2f-1)d$}\\
\text{$(2f+1)d<(2f+2)d<\cdots<(n)d$}\\
\end{matrix}\Biggr\}.
$$ For each $\lambda\in\pp_{n-2f}$, we denote by
$\std_{2f}(\lambda)$ the set of all the standard $\lambda$-tableaux
with entries in $\{2f+1,\cdots,n\}$. The initial tableau
$\t_f^{\lam}$ in this case has the numbers $2f+1,\cdots,n$ in order
along successive rows. Again, for each $\t\in\Std_{2f}(\lambda)$,
let $d(\t)$ be the unique element in
$\Sym_{\{2f+1,\cdots,n\}}\subseteq\Sym_n$ with
$\t_f^{\lam}d(\t)=\t$. Let $\sigma\in\Sym_{\{2f+1,\cdots,n\}}$ and
$d_1,d_2\in\mathcal{D}_{f}$. Then $d_1^{-1}e_1e_3\cdots
e_{2f-1}\sigma d_2$ corresponds to the Brauer $n$-diagram where the
top horizontal edges connect $(2i-1)d_1$ and $(2i)d_1$, the bottom
horizontal edges connect $\bigl((2i-1)d_2\bigr)^{-}$ and $\bigl((2i)d_2\bigr)^{-}$, for
$i=1,2,\cdots,f$, and the vertical edges connects $(j)d_1$ with $
\bigl((j)d_2\bigr)^-$ for $j=2f+1,2f+2,\cdots,n$.

\begin{lemma} {\rm (\cite[Corollary 3.3]{DDH})} With the above notations, the set
$$ \biggl\{d_1^{-1}e_1e_3\cdots e_{2f-1}\sigma
d_2\biggm|\text{$0\leq f\leq [n/2]$,
$\sigma\in\Sym_{\{2f+1,\cdots,n\}}$, $d_1,
d_2\in\mathcal{D}_{f}$}\biggr\}.$$ is a basis of the Brauer algebra
$\bb_{n}(x)_R$, which coincides with the natural basis given by
Brauer $n$-diagrams.
\end{lemma}

\begin{definition} \label{dfn000} Let $D=d_1^{-1}e_1e_3\cdots e_{2f-1}\sigma d_2\in\Bd_n$, where
$0\leq f\leq [n/2]$, $\sigma\in\Sym_{\{2f+1,\cdots,n\}}$, $d_1,
d_2\in\mathcal{D}_{f}$. Then we define
$\ell(D):=\ell(d_1)+\ell(d_2)+\ell(\sigma)$ and
$\sign(D):=(-1)^f(-1)^{\ell(D)}$.
\end{definition}

\begin{remark} \label{rm} 1) We can always draw the Brauer diagram
as a ``nice diagram" (i.e., in a way such that two edges intersect
at most once and there are no self-intersections and no three edges
intersect at one point, etc, see \cite[1.1]{FG}). If $D$ is
represented by a ``nice diagram" with $f$ horizontal edges in each
row and $n(D)$ is the number of crossings of edges, then
$\sign(D)=(-1)^{f+n(D)}$. Moreover, if $D\in\Sym_n$ then $n(D)$
coincides with the length function on $\Sym_n$ which we introduced
before.

2) For any $D_1,D_2\in\Bd_n$, note that in general
$$\sign(D_1D_2)\neq\sign(D_1)\sign(D_2). $$

3) Our definition of $\sign(D)$ coincides with that of
$\varepsilon(D)$ in \cite[1.4]{Ga2}.
\end{remark}
\smallskip

Note that, however, the above basis is not a cellular basis for
$\bb_n(x)$. But if we replace the $\sigma$ in the above basis by a
Murphy basis element of $R\Sym_{\{2f+1,\cdots,n\}}$ then we will get
a cellular basis of $\bb_n(x)$. Precisely, the set $$
\Biggl\{d_1^{-1}e_1e_3\cdots
e_{2f-1}\bigl(d(\s)^{-1}X_{\lam}^{(f)}d(\t)\bigr)
d_2\Biggm|\begin{matrix}\text{$0\leq f\leq [n/2]$,
$\lam\in\pp_{n-2f}$}\\
\text{$\s,\t\in\std_{2f}(\lam)$,$d_1,
d_2\in\mathcal{D}_{f}$}\end{matrix}\Biggr\},$$ where
$X_{\lam}^{(f)}:=\sum_{w\in\Sym^{(f)}_{\lam}}w$ and
$$\Sym^{(f)}_{\lam}:=\Sym_{\{2f+1,\cdots,2f+\lam_1\}}\times\Sym_{\{2f+\lam_1+1,\cdots,2f+\lam_1+\lam_2\}}\times\cdots,
$$ is a cellular basis of the Brauer algebra $\bb_{n}(x)_R$. The cell modules of $\bb_{n}(x)_R$ are labeled by the set of pairs $(f,\lam)$, where $0\leq f\leq [n/2]$ and $\lam\vdash n-2f$. For any
two pairs $(f,\lam), (g,\mu)$ with $0\leq f,g\leq [n/2]$ and
$\lam\in\pp_{n-2f},\mu\in\pp_{n-2g}$, we define $(f,\lam)\unrhd
(g,\mu)$ if either $f>g$ or $f=g$ and $\lam\unrhd\mu$. If
$(f,\lambda)\unrhd(g,\mu)$ and $(f,\lambda)\neq(g,\mu)$, then we
write $(f,\lambda)\rhd(g,\mu)$. We use $(\bb_n(x))^{\unrhd(f,\lam)}$
(respectively, $(\bb_n(x))^{\rhd(f,\lam)}$) to denote the free
$R$-submodule of $\bb_n(x)$ spanned by the cellular basis elements
corresponding to those $(g,\mu,d_1,d_2,\s,\t)$ with
$\mu\in\pp_{n-2g}$, $d_1,d_2\in\mathcal{D}_g$, $\s,\t\in\std(\mu)$
and $(g,\mu)\unrhd(f,\lam)$ (respectively, $(g,\mu)\rhd(f,\lam)$).
Then both $(\bb_n(x))^{\unrhd(f,\lam)}$ and
$(\bb_n(x))^{\rhd(f,\lam)}$ are two-sided ideals of $\bb_n(x)$. In
particular, if we denote by $\bb_n(x)^{(f)}$ the two-sided ideal of
$\bb_n(x)$ generated by $e_1e_3\cdots e_{2f-1}$, then $$
\bb_n(x)^{(f)}=\sum_{\lam\vdash n-2f}(\bb_n(x))^{\unrhd(f,\lam)}
$$
is spanned by the cellular basis elements which it contains.
Henceforth, we shall write $\bb_n^{(f)}$ instead of $\bb_n(x)^{(f)}$
for simplicity.

The Brauer algebra $\bb_n(x)$ and its specialized version have been
studied in a number of references, e.g., \cite{Br}, \cite{B1},
\cite{B2}, \cite{CV}, \cite{CVM1}, \cite{CVM2}, \cite{DP},
\cite{DDH}, \cite{DWH}, \cite{FG}, \cite{Ga1}, \cite{Ga2},
\cite{Hu}, \cite{HX1}, \cite{Ma}, \cite{RS} and \cite{W}. In the set up of
Schur--Weyl duality for orthogonal groups, we only need certain
specialized Brauer algebras. Let $K$ be a field of characteristic
not equal to $2$. Let $m\in\N$ and $V$ an $m$-dimensional orthogonal
space over $K$. Let $\bb_n(m)_{\Z}:=\Z\otimes_{\Z[x]}\bb_n(x)$,
where $\Z$ is regarded as $\Z[x]$-algebra by specifying $x$ to $m$.
Let $\bb_n(m):=K\otimes_{\Z}\bb_n(m)_{\Z}$, where $K$ is regarded as
$\Z$-algebra in the natural way. Then there is a right action of
$\bb_n(m)$ on the $n$-tensor space $V^{\otimes n}$ which commutes
with the natural left action of $O(V)$. We recall the definition of
this action. Let $\delta_{i,j}$ denote the value of the usual
Kronecker delta. We fix an ordered basis
$\bigl\{v_1,v_2,\cdots,v_{m}\bigr\}$ of $V$ such that
$$ (v_i, v_{j})=\delta_{i, m+1-j},\quad\forall\,\,1\leq i, j\leq m.$$
The right action of $\bb_n(m)$ on $V^{\otimes n}$ is defined on
generators by
$$\begin{aligned} (v_{i_1}\otimes\cdots\otimes
v_{i_n})s_j&:=v_{i_1}\otimes\cdots\otimes v_{i_{j-1}}\otimes
v_{i_{j+1}}\otimes v_{i_{j}}\otimes v_{i_{j+2}}
\otimes\cdots\otimes v_{i_n},\\
(v_{i_1}\otimes\cdots\otimes
v_{i_n})e_j&:=\delta_{i_{j},m+1-i_{j+1}} v_{i_1}\otimes\cdots\otimes
v_{i_{j-1}}\otimes\biggl(
\sum_{k=1}^{m}v_{k}\otimes v_{m+1-k}\biggr)\\
& \qquad\qquad\otimes v_{i_{j+2}}\otimes\cdots\otimes
v_{i_n}.\end{aligned}
$$

Let $\overline{K}$ be the algebraic closure of $K$. Set
$V_{\overline{K}}:=\overline{K}\otimes_{K}V$. Then by the main
results in \cite{Br}, \cite{B1}, \cite{B2}, \cite{DP}, \cite{DH} and
\cite{We}, we know that there is a Schur--Weyl duality between
$\bb_n(m)_{\overline{K}}$ and $O(V_{\overline{K}})$ on
$V_{\overline{K}}^{\otimes n}$. In particular, we have two
surjective homomorphisms as follows: $$ \varphi_{\overline{K}}:
\bb_n(m)_{\overline{K}}\rightarrow\End_{O(V_{\overline{K}})}(V_{\overline{K}}^{\otimes
n}),\quad \psi_{\overline{K}}:
\overline{K}O(V_{\overline{K}})\rightarrow\End_{\bb_n(m)_{\overline{K}}}(V_{\overline{K}}^{\otimes
n}).
$$
Furthermore, $\varphi_{\overline{K}}$ is injective if and only if
$m\geq n$. If $m<n$ then $\dim\Ker\varphi_{\overline{K}}$ is
independent of the characteristic of the field $K$ (as long as $\ch
K\neq 2$).

\smallskip
{\it From now on until the end of this section, we assume that
$m<n$}. The main results in \cite{DH} actually implies that
$\dim_{\overline{K}}\Ker{\varphi_{\overline{K}}}=\dim_{K}\Ker{\varphi}$
and
$\Ker{\varphi_{\overline{K}}}=\overline{K}\otimes_{K}\Ker{\varphi}$
because \cite[Theorem 1.4, Theorem 6.9]{DH} gave an integral basis
for $\Ker{\varphi_{\overline{K}}}$. In particular, $\dim\Ker\varphi$
is independent of the characteristic of the field $K$ (as long as
$\ch K\neq 2$). In the following sections we shall sometimes use
$\Ann_{\mathfrak{B}_n(m)}\bigl(V^{\otimes n}\bigr)$ to denote the
annihilator of $V^{\otimes n}$ in $\bb_n(m)$. By definition,
$\Ann_{\mathfrak{B}_n(m)}(V^{\otimes n})=\Ker\varphi$.\medskip

\section{The annihilator of $n$-tensor space}\label{xxsec3}

In this section, we shall prove that the annihilator of $V^{\otimes n}$ in
$\bb_n(m)$ is equal to the two-sided ideal generated by
$E_0, E_1,\dots,E_{[(m+1)/2]}$. This generalizes the earlier
result (for the case $\ch K=0$) of Lehrer--Zhang \cite[Theorem
6.1]{LZ2}.\medskip

We first recall a definition and a result given in \cite{LZ2}.

\begin{definition} {\rm (\cite[Lemma 4.1]{LZ2})} \label{lzdfn1} Let $S=(i_1,\cdots,i_N), S'=(j_1,\cdots,j_N)$ be two $N$-tuples of integers such that $\{i_1,\cdots,i_N\}, \{j_1,\cdots,j_N\}$ are two disjoint subsets of $\{1,2,\cdots,2n\}$. Let $\beta$ be any pairing of the vertices $\bigl\{1,2,\cdots,2n\bigr\}\setminus\{i_1,\cdots,i_N,j_1,\cdots,j_N\}$. For $w\in\mathfrak{S}_N$, let $D_{w}(S,S',\beta)$ be the Brauer diagram with edges $\bigl\{(i_k,j_{\pi(k)})\bigm|k=1,2,\cdots,N\bigr\}\sqcup\beta$. We define $$
b(S,S',\beta):=\sum_{w\in\mathfrak{S}_N}\sign(w)D_{w}(S,S',\beta)\in\bb_n(m).
$$
\end{definition}

\begin{lemma} {\rm (\cite[(4.5)]{LZ2})} \label{lzdfn} Let $S_{a,b}:=(1,2,\cdots,a,(a+1)^-,(a+2)^-,\cdots,(a+b)^-)$,
$S'_{a,b}:=(a+1,a+2,\cdots,a+b,1^-,2^-,\cdots,a^-)$, $\beta_{a+b}$ be the pairing $(a+b+1,(a+b+1)^-)$, $(a+b+2,(a+b+2)^-),\cdots,(n,n^-)$.
Then we have that $$
E_{a,b}=b(S_{a,b}, S'_{a,b}, \beta_{a+b}).
$$
\end{lemma}

The advantage of the above alternative description of $E_{a,b}$ lies in that the sign $\sign(w)$ before $D_{w}(S,S',\beta)$ depends only on $w$ which is more easier to be handled than the sign $\sign(D_{w}(S,S',\beta))$. More precisely, up to a sign, $b(S,S',\beta)$ depends only on $\beta$ and the two subsets $\{i_1,\cdots,i_N\}$, $\{j_1,\cdots,j_N\}$ but not on the orderings on these two subsets.
\medskip

For any $h\in\bb_n(m)$, we use $\<h\>$ to denote the two-sided ideal
of $\mathfrak{B}_n(m)$ generated by $h$. For any finite set $S$, we use $|S|$ to denote the cardinality of $S$.

\begin{lemma} \label{lm00} Let $a,b\in\N$ such that $1\leq a+b\leq n$. Then there exist two elements $w_1, w_2\in
\Sym_{a+b}$ such that
$E_{b,a}=\pm w_1E_{a,b}w_2$. In particular,
$\<E_{a,b}\>=\<E_{b,a}\>$.
\end{lemma}

\begin{proof} This is clear by Lemma \ref{lzdfn}. In fact, we can take $w_1$ to be the Brauer
$n$-diagram which has $$\begin{aligned}
&\{1,(a+1)^-\},\hspace{6pt} \{2,(a+2)^-\},\cdots,\{b,(a+b)^-\},\\
&\{b+1,1^-\},\hspace{6pt} \{b+2,2^-\},\cdots,\{b+a,a^-\},\\
&\{r,r^-\},\hspace{10pt}\text{for all}\,\,r\in\{a+b+1,a+b+2,\cdots,n\},
\end{aligned}$$
as its (vertical) edges; and $w_2$ to be the Brauer
$n$-diagram which has $$\begin{aligned}
&\{1,(b+1)^-\},\hspace{6pt} \{2,(b+2)^-\},\cdots,\{a,(b+a)^-\},\\
&\{a+1,1^-\},\hspace{6pt} \{a+2,2^-\},\cdots,\{a+b,b^-\},\\
&\{r,r^-\},\hspace{10pt}\text{for all}\,\,r\in\{a+b+1,a+b+2,\cdots,n\},
\end{aligned}$$
as its vertical edges.
\end{proof}

The proof of the next lemma uses the original definition of $E_{a,b}$.

\begin{lemma}\label{xx3.1} For any positive integers $a, b$ with $1\leq a+b\leq n$, we have that $E_{a,b}\in
\langle E_{a,b-1}\rangle\cap\langle E_{a-1,b}\rangle$.
\end{lemma}

\begin{proof} For each $k\in\{1,2,\cdots,a,(a+1)^-,\cdots,(a+b)^-\}$, we use $\Bd(k;a,b)$ to denote the subset of
the Brauer diagrams in $\Bd(a,b)$ which have the edge $\{k,a+b\}$.

If $k=i\in\{1,2,\cdots,a\}$ then we use $d$ to denote the Brauer
$n$-diagram which has $\{k,a+b\}$ and $\{k^-,(a+b)^-\}$ as its only
horizontal edges and
$$
\{r,r^-\},\hspace{10pt}\text{for all}\,\,r\in\{1,2,\cdots,n\}\backslash
\{k,a+b\},
$$
as its vertical edges. It is clear that $\sign(d)=-1$. By the
concatenation rule of Brauer diagrams, it is easy to see that $$
\bigl\{D\bigm|D\in\Bd(k;a,b)\bigr\}=\bigl\{dD''\bigm|D''\in\Bd(a,b-1)\bigr\}.
$$
We claim that $$
\sum_{D\in\Bd(k;a,b)}\sign(D)D=-dE_{a,b-1}\in\langle
E_{a,b-1}\rangle.
$$
To prove this claim, it suffices to show that for each
$D''\in\Bd(a,b-1)$,
\begin{equation}\label{2sign}\sign(dD'')=\sign(d)\sign(D'')=-\sign(D'').
\end{equation}

Note that when concatenating a ``nice diagram" for $d$ with a ``nice diagram" for $D''$ and transforming it into a ``nice diagram" for $dD''$, the
only transformation which can possibly change the parity of $\ell(d)+1+\ell(D'')+f$ (where $2f$ is the number of horizontal edges of $D''$) is for the following type of edge which was drawn in red color
(where $D''\in\Bd(a,b-1)$):
\medskip
\begin{center}
\scalebox{0.5}[0.5]{\includegraphics{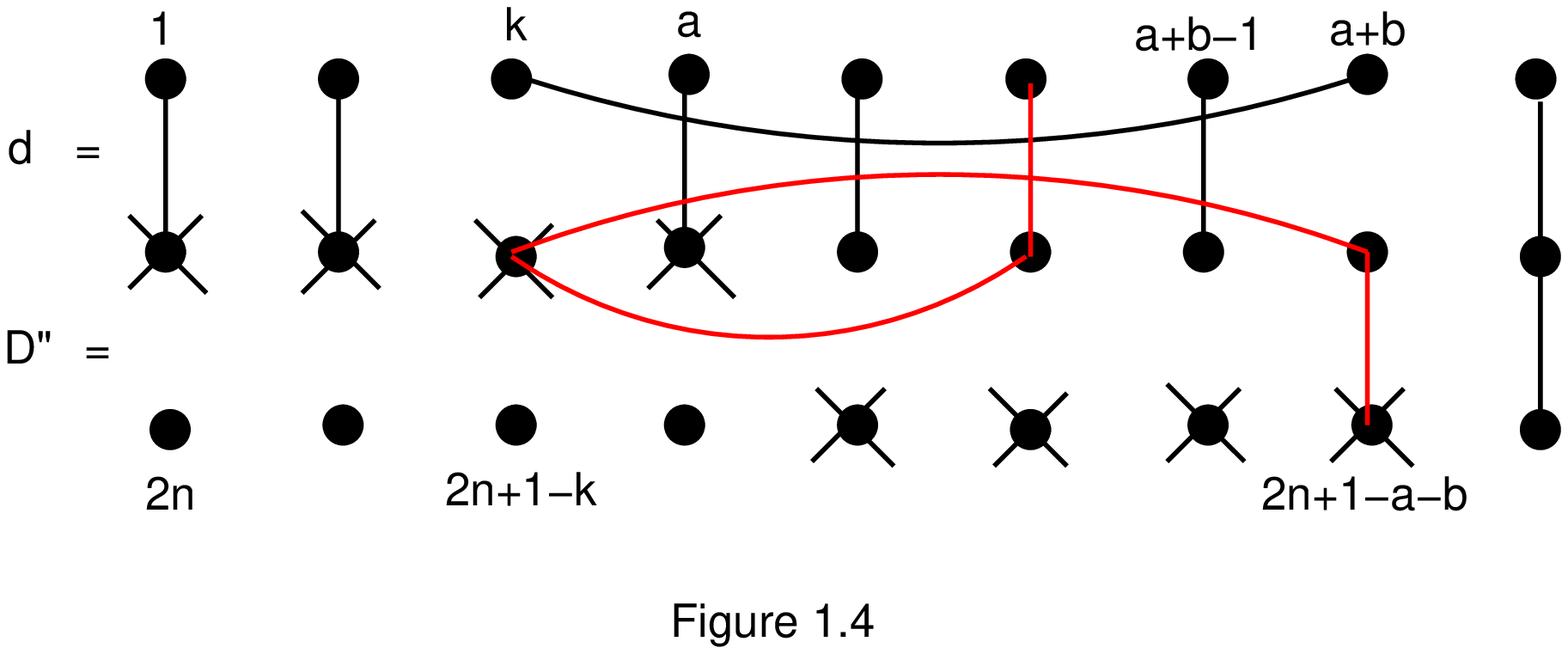}}
\end{center}
\medskip
That is, we need to eliminate the self-intersection in the following picture.
\medskip
\begin{center}
\scalebox{0.28}[0.28]{\includegraphics{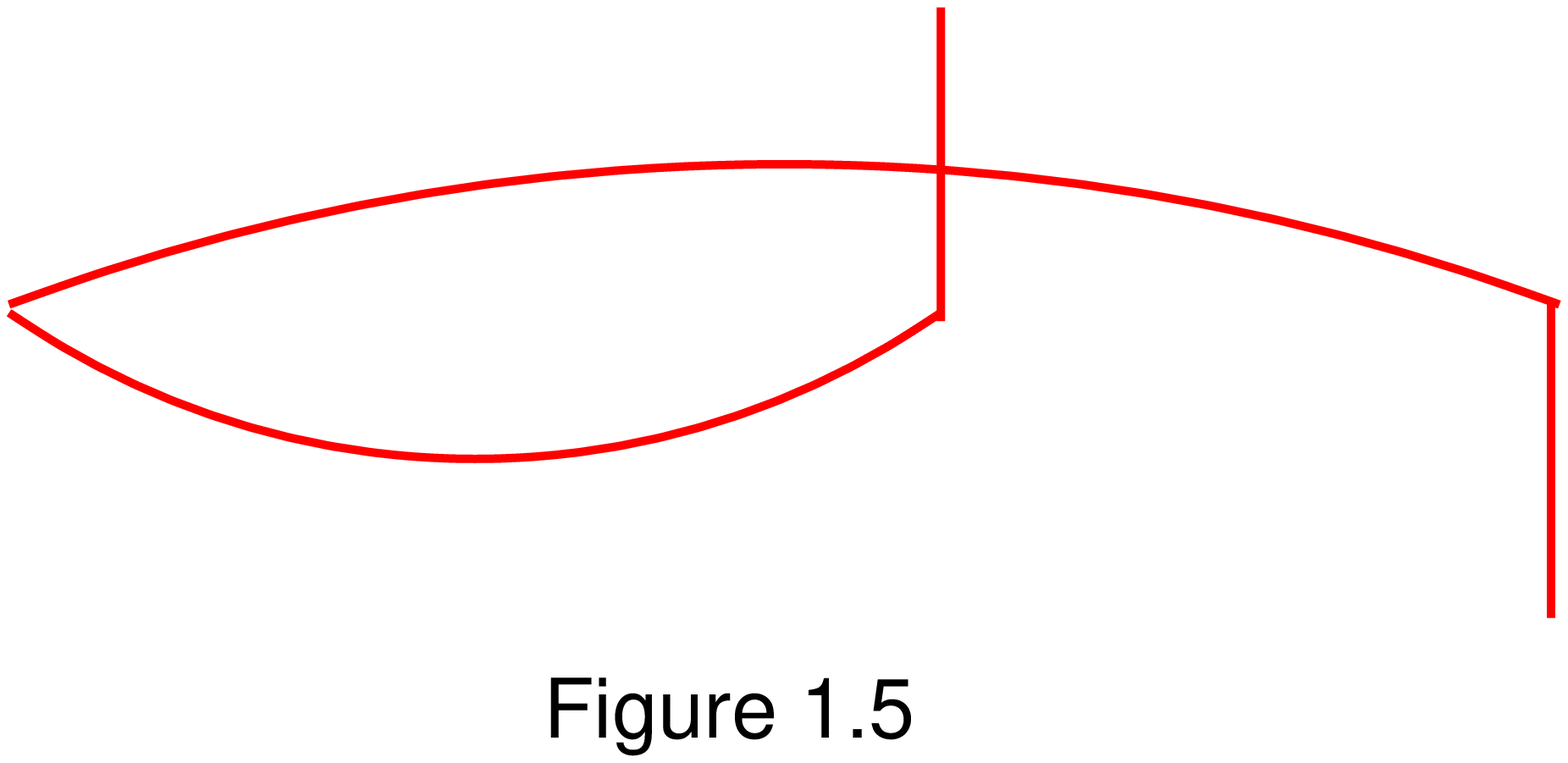}}
\end{center}
\medskip
However, by eliminating the above self-intersection and making it into an edge in a ``nice diagram" for $dD''$
has the effect of removing one horizontal edge on the top rows of $D''$ together with eliminating
$2k-1$ crossing on this concatenation diagram for some $k\in\N$. To be more precise, when we eliminate the self-intersection in Figure 1.5, 
the immediate effect is that we will remove one top horizontal edge of $D''$ as well as one crossing from the concatenation diagram. However, there are possibly some more crossings which will be removed. These are the crossings arising from the vertices inside the area circled
by the edge in Figure 1.5. If a vertex $\gamma$ inside the area connects with another vertex which is also inside the circled area then
these two interior vertices will contribute two crossings with the red line which will finally be eliminated; otherwise
$\gamma$ itself will connect with two different vertices outside the circled area and hence will produce two crossings with
the red line which will finally be eliminated. To sum all, the sign finally remains unchanged. This proves (\ref{2sign}). \smallskip

If $k=i^-\in\{(a+1)^-,(a+2)^-,\cdots,(a+b)^-\}$ then we use $d'$ to denote the Brauer $n$-diagram which has
$$\begin{aligned}
&\{i,(i+1)^-\},\hspace{6pt} \{i+1,(i+2)^-\},\cdots,\{a+b-1,(a+b)^-\},\hspace{6pt}\{a+b,i^-\},\\
&\{r,r^-\},\hspace{10pt}\text{for all}\,\,r\in\{1,2,\cdots,i-1\}\sqcup \{a+b+1,a+b+2,\cdots,n\},
\end{aligned}$$
as its (vertical) edges. It is clear that $\sign(D_2)=(-1)^{a+b-i}$. Then by a similar argument as in the case $k=i$, we can deduce that
$$
\sum_{D\in\Bd(k;a,b)}\sign(D)D=(-1)^{a+b-i}E_{a,b-1}d'\in\langle
E_{a,b-1}\rangle.
$$
Therefore, we have that
$$
E_{a,b}=\sum_{k\in\{1,2,\cdots,a,(a+1)^-,\cdots,(a+b)^-\}}\sum_{D\in\Bd(k;a,b)}\sign(D)D\in\langle
E_{a,b-1}\rangle.
$$
This proves $E_{a,b}\in \langle E_{a,b-1}\rangle$. It remains to
prove that $E_{a,b}\in \langle E_{a-1,b}\rangle$.

Exchanging the roles of $a$ and $b$ and using Lemma \ref{lm00}, we see that $$
\<E_{a,b}\>=\<E_{b,a}\>\subseteq\<E_{b,a-1}\>=\<E_{a-1,b}\>.
$$
as required. This completes the proof of the lemma.
\end{proof}

Note that if $a\geq 1$ and $b=0$ (respectively, if $a=0$ and $b\geq 1$) then, by the theory of
symmetric group, we have that $E_{a,0}\in\<E_{a-1,0}\>$ (respectively, $E_{0,b}\in\<E_{0,b-1}\>$).
\medskip

Let $A^{(1)}, A^{(2)}, A^{(3)}, A^{(4)}$ be four subsets of indices such that \begin{enumerate}
\item $A^{(i)}\cap A^{(j)}=\emptyset$ for any $1\leq i\neq j\leq 4$; and
\item $A^{(1)}\sqcup A^{(3)}\subseteq \{1,2,\cdots,n\}$, $A^{(2)}\sqcup A^{(4)}\subseteq \{1^{-},2^{-},\cdots,n^{-}\}$; and
\item $|A^{(1)}|+|A^{(2)}|=|A^{(3)}|+|A^{(4)}|$.
\end{enumerate}
Recall that $i^{-}=2n+1-i$ for each $1\leq i\leq 2n$. We set $n_0:=|A^{(1)}|+|A^{(2)}|$, and
$$
\bigl\{a_1,a_2,\cdots,a_{2n-2n_0}\bigr\}:=\bigl\{1,2,\cdots,2n\bigr\}\setminus \bigsqcup_{k=1}^4 A^{(k)}.
$$
Let $(i_1,j_1,i_2,j_2,\cdots,i_{n-n_0},j_{n-n_0})$ be a fixed
permutation of $\{a_1,\cdots,a_{2n-2n_0}\}$. Set
$$\begin{aligned}
& \mathbf{i}:=(i_1,i_2,\cdots,i_{n-n_0}),\hspace{10pt} \mathbf{j}:=(j_1,j_2,\cdots,j_{n-n_0}).\\
& a_{11}:=|A^{(1)}|,\quad a_{12}:=|A^{(2)}|.
\end{aligned}
$$
Let $\beta_{\mathbf{i},\mathbf{j}}$ be the pairing $(i_1,j_1),(i_2,j_2),\cdots,(i_{n-n_0},j_{n-n_0})$. We fix an ordering on $A^{(1)}\sqcup A^{(2)}$ and an ordering on $A^{(3)}\sqcup A^{(4)}$ respectively. We define $S_A, S'_A$ to be the corresponding $n_0$-tuples with respect to the two orderings. As we said before, for different choices of orderings, $b(S_A,S'_A,\beta_{\mathbf{i},\mathbf{j}})$ differs only by a sign.

\begin{lemma}\label{xx3.3}
With notations as above, we have that $$
\sum_{w\in\mathfrak{S}_{n_0}}b(S_A,S'_{A},\beta_{\mathbf{i},\mathbf{j}})\in\langle E_{a_{11},a_{12}}\rangle=\langle E_{a_{12},a_{11}}\rangle.
$$
\end{lemma}

\begin{proof} Assume that $$\begin{aligned}
& S_A=(q_1,\cdots,q_{a_{11}},p_{a_{11}+1}^-,\cdots,p_{a_{11}+a_{12}}^-),\\
& S'_A=(q_{a_{11}+1},\cdots,q_{a_{11}+a_{12}},p_1^-,\cdots,p_{a_{11}}^-),\\
& \{q_{n_0+1},\cdots,q_n\}=\{i_1,\cdots,i_{n-n_0},j_1,\cdots,j_{n-n_0}\}\cap\{1,2,\cdots,n\},\\
& \{p_{n_0+1}^-,\cdots,p_n^-\}=\{i_1,\cdots,i_{n-n_0},j_1,\cdots,j_{n-n_0}\}\cap\{1^-,2^-,\cdots,n^-\}.
\end{aligned}$$

We use $\sigma_1$ to denote the Brauer diagram which has the following edges $$\begin{aligned}
& \{k, q_k^-\},\quad\,\text{for $1\leq k\leq n_0=a_{11}+a_{12}$},\\
& \{n_0+l, q_{n_0+l}^-\},\quad\,\text{for $1\leq l\leq n-n_0$,}
\end{aligned}
$$
and use $\sigma_2$ to denote the Brauer diagram which has the following edges $$\begin{aligned}
& \{p_k, k^-\},\quad\,\text{for $1\leq k\leq n_0=a_{11}+a_{12}$},\\
& \{p_{n_0+l},(n_0+l)^-\},\quad\,\text{for $1\leq l\leq n-n_0$}.
\end{aligned}
$$
Then $\sigma_1, \sigma_2$ are both elements in the symmetric group $\mathfrak{S}_n$.

The pairing $\beta_{\mathbf{i},\mathbf{j}}$ and the elements $\sigma_1, \sigma_2$ determine a pairing $\beta$ on the set of vertices $\{n_0+1,n_0+2,\cdots,n,(n_0+1)^-,(n_0+2)^-,\cdots,n^-\}$, and hence a Brauer $(n-n_0)$-diagram $D$. Since the number of top horizontal edges of $D$ is the same as the number of the bottom horizontal edges of $D$, we can clearly write $D=D_1D_0D_2$ such that $D_1,D_0,D_2\in\Bd_{n-n_0}$ and
$D_0$ contains only the vertical edges of the form $(k,k^-)$ with $n_0+1\leq k\leq n$. We extend the Brauer diagrams $D_1,D_2\in\Bd_{n-n_0}$ to be Brauer diagrams $D'_1,D'_2\in\Bd_n$ by adding the vertical edges $(k,k^-)$ with $1\leq k\leq n_0$ to their left-hand sides.

Then it follows directly from the definition in (\ref{lzdfn}) that $$
\sigma_1b(S_A,S'_{A},\beta_{\mathbf{i},\mathbf{j}})\sigma_2=\pm D'_1{E}_{a_{11},a_{12}}D'_2.
$$

It follows (by Lemma \ref{lm00}) that $$
b(S_A,S'_{A},\beta_{\mathbf{i},\mathbf{j}})=\pm \sigma_1^{-1}D'_1{E}_{a_{11},a_{12}}D'_2\sigma_2^{-1}\in\langle E_{a_{11},a_{12}}\rangle=\langle E_{a_{12},a_{11}}\rangle,
$$
as required. Hence we complete the proof of the lemma.
\end{proof}

For the sake of simplicity, we shall abbreviate the partition
$(\underbrace{a,\cdots,a}_{\text{$k$ copies}})$ as $(a^k)$.

\begin{definition} {\rm (\cite[Theorem 1.4]{DH})} \label{xx3.4} We set
$$\begin{aligned} (2\mathcal{P}_n)'&:=\bigl\{\tilde{\lambda}:=(\lambda_1,\lambda_1,\lambda_2,\lambda_2,\cdots)\vdash 2n\bigm|\lambda=(\lambda_1,\lambda_2,\cdots)\in \mathcal{P}_n\bigr\},\\
T_m&:=\bigl\{(\nu,\mathfrak{t})\bigm|\mathfrak{t}\in\Std(\tilde{\nu}), (m+1,1^{n-m-1})\unlhd \nu\in \mathcal{P}_n\bigr\}.
\end{aligned} $$
\end{definition}

Now we are in a position to state the main result of this section.

\begin{theorem}\label{xx3.6}
Let $K$ be an arbitrary field of characteristic other than two. If $n>m$, then
$$
\Ann_{\mathfrak{B}_n(m)}\left(V^{\otimes n}\right)=\left\langle E_0,E_1,\cdots,E_{[\frac{m+1}{2}]} \right\rangle.
$$
\end{theorem}

\begin{proof}
By \cite[Theorem 1.4 and Theorem 6.9]{DH}, we know that $\Ann_{\mathfrak{B}_n(m)}\left(V^{\otimes n}\right)$ has
a basis consisting of elements of the form $Y_{\nu,\mathfrak{t}}$, where $(\nu,\mathfrak{t})\in T_m$. It remains to show (by the first statament of Theorem \ref{LZ}) that each $Y_{\nu,\mathfrak{t}}$ belongs the two-sided ideal generated by $E_0,E_1, \cdots,E_{[\frac{m+1}{2}]}$.

Let $(\nu,\mathfrak{t})\in T_m$. By the definition (see \cite[\S6]{DH})
$$
Y_{\nu,\mathfrak{t}}=(-1)^{\ell(d(\mathfrak{t}))}Y_{\nu}\ast d(\t)=(-1)^{\ell(d(\mathfrak{t}))}\sum_{\mathbf{i},\mathbf{j}}(\pm b(S_A,S'_{A},\beta_{\mathbf{i},\mathbf{j}})),
$$
where $``\ast"$ denotes the permutation action of $\mathfrak{S}_{2n}$ on $\bb_n(m)$ (see \cite[Section 6]{DH}), and
$$\begin{aligned}
A^{(1)}:&=\{(i)d(\mathfrak{t})\ |\ i=1,2,3,\cdots,\nu_1\}\cap \{1,2,\cdots,n\},\\
A^{(2)}:&=\{(i)d(\mathfrak{t})\ |\ i=1,2,3,\cdots,\nu_1\}\cap \{1^{-},2^{-},\cdots,n^{-}\},\\
A^{(3)}:&=\{(i^-)d(\mathfrak{t})\ |\ i=1,2,3,\cdots,\nu_1\}\cap \{1,2,\cdots,n\},\\
A^{(4)}:&=\{(i^-)d(\mathfrak{t})\ |\ i=1,2,3,\cdots,\nu_1\}\cap \{1^{-},2^-,\cdots,n^{-}\},\\
\end{aligned}$$
with $|A^{(1)}|+|A^{(2)}|=|A^{(3)}|+|A^{(4)}|=\nu_1$, $S_A, S'_A$ are defined by using certain prefixed ordering on the sets $A^{(1)}\sqcup A^{(2)}$,
$A^{(3)}\sqcup A^{(4)}$respectively, and
$$
\mathbf{i}:=(i_1,i_2,\cdots,i_{n-\nu_1}),\hspace{10pt} \mathbf{j}:=(j_1,j_2,\cdots,j_{n-\nu_1})
$$
such that $(i_1,j_1,i_2,j_2,\cdots,i_{n-\nu_1},j_{n-\nu_1})$ runs over a subset of permutations of
the integers in $\{1,2,\cdots,2n\}\setminus \bigsqcup_{k=1}^4 A^{(k)}$.

By Lemma \ref{xx3.3}, we obtain
$$
b(S_A,S'_{A},\beta_{\mathbf{i},\mathbf{j}})\in
\left\langle E_{|A^{(1)}|,|A^{(2)}|} \right\rangle.
$$
Note that the condition $(\nu,\mathfrak{t})\in T_m$ implies that $|A^{(1)}|+|A^{(2)}|=\nu_1\geq m+1$.
It follows from Lemma \ref{xx3.1} that $b(S_A,S'_{A},\beta_{\mathbf{i},\mathbf{j}})$
belongs to the two-sided ideal generated by $E_0,E_1,\cdots,E_{m+1}$.

On the other hand, it is clear that for any integer $[\frac{m+1}{2}]+1\leq i\leq m+1$, there exists $\sigma'_1, \sigma'_2\in\Sym_n$, such that $E_{i}=\pm \sigma'_1 E_{m+1-i}\sigma'_2\in \langle E_{m+1-i} \rangle$. As a consequence, we get that $$
b(S_A,S'_{A},\beta_{\mathbf{i},\mathbf{j}})\in\<E_0, E_1,\cdots, E_{[(m+1)/2]}\>,
$$
as required. This completes the proof of the theorem.
\end{proof}

\bigskip\bigskip

\section{Proof of Theorem \ref{conj}}

In this section we shall give the main result of this paper. That is, the proof of Theorem \ref{conj}.\smallskip

We shall proceed the proof in three steps. The first step is to prove a statement (Theorem \ref{mainthm}) about identification between certain two-sided ideals in the symmetric group algebra $K\Sym_n$. To this end, we need to recall the seminormal basis (\cite{Ho}, \cite{DJ2}, \cite{M:gendeg}) of the symmetric group algebra. We shall follow the approach in \cite{M:gendeg}. Note that \cite{M:gendeg} only consider the seminormal basis of the
(cyclotomic) Hecke algebra with $q\neq 1$. The symmetric group case
(i.e., $q=1$) is similar and may be proved using the same arguments.
The only real difference between the cases $q\neq 1$ and $q=1$ is
the choice of content function: if $q\neq 1$ then
$\cont_\v(k)=\xi^{c-r}$, when $\v(r,c)=k$, and if $q=1$ then,
instead, $\cont_\v(k)=c-r$. Analogous minor `logarithmic'
adjustments are required in the argument below when $q=1$.\smallskip

Set $L_1:=0$ and define $L_{i+1}:=s_iL_is_i+s_i$ for
$i=1,\cdots,n-1$. The elements $L_1,\cdots,L_n$ are called the
Jucys--Murphy operators of the symmetric group $\Sym_n$. Let
$\lam\vdash n$ and $\t\in\Std(\lam)$. For any integer $1\leq k\leq
n$, we define $\cont_{\t}(k)=j-i$ if $k$ appears in row $i$ and
column $j$ in $\t$. Let $$
\mathcal{R}(k):=\bigl\{d\bigm|\text{$|d|<k$ and $d\neq 0$ if
$k=2,3$}\bigr\},
$$
which is the complete set of possible contents $\cont_{\t}(k)$ as $\t$
runs over the set of standard tableaux.

\begin{definition} {\rm (\cite[Definition 2.4]{M:gendeg})} Let $\lam\vdash n$ and $\s,\t\in\Std(\lam)$. \begin{enumerate}
\item[(i)] Let $F_{\t}:=\prod\limits_{k=1}^n\prod\limits_{\substack{c\in\mathcal{R}(k)\\ \cont_{\t}(k)\neq c}}\dfrac{L_k-c}{\cont_{\t}(k)-c}$.
\item[(ii)] Let $f_{\s\t}:=F_{\s}X_{\s\t}F_{\t}$.
\end{enumerate}
\end{definition}

Let $\lam\vdash n$ and $\t\in\Std(\lam)$. For each integer $1\leq
k\leq n$ we use $\t_{k}$ to denote the subtableau of $\t$ which
contains $\{1,2,\cdots,k\}$. If $\gamma=(i,j)\in[\lam]$ such that
$[\lam]\setminus\{\gamma\}$ is again the Young diagram of a
partition $\mu$. Then we call $\gamma$ a removable node of $\lam$
and an addable node of $\mu$. For any two nodes $\gamma=(i,j),
\gamma'=(i',j')$ we say that $\gamma$ is below $\gamma'$, or
$\gamma'$ is above $\gamma$ if $i>i'$.

\begin{definition} {\rm (\cite[(2.8)]{M:gendeg}\footnote{We remark that there is a typos in
\cite[Page 704, Line 9]{M:gendeg}, $y\prec x$ there should be replaced by $y\succ x$, cf.
\cite[3.15]{JM:cyc-Schaper}.}, \cite[3.15]{JM:cyc-Schaper})} Let $\lam\vdash n$ and $\t\in\Std(\lam)$.
For $k=1,\cdots,n$, let $\Add_\t(k)$ be the set of addable nodes of the partition $\Shape({\t}_k)$ which are \textit{below} $\t^{-1}(k)$. Similarly, let $\Rem_\t(k)$ be the set of removable nodes of $\Shape({\t}_k)$ which are \textit{below} $\t^{-1}(k)$. Now define
  $$\gamma_\t=\prod_{k=1}^n\dfrac{\prod_{\alpha\in\Add_\t(k)}
                      \bigl(\cont_\t(k)-\cont(\alpha)\bigr)}
                  {\prod_{\rho\in\Rem_\t(k)}
                   \bigl(\cont_\t(k)-\cont(\rho)\bigr)}\quad\in\Q,
  $$
and $\tilde{f}_{\s\t}:=\gamma_{\t}^{-1}f_{\s\t}$ for any $\s\in\Std(\lam)$.
\end{definition}

\begin{lemma}\label{lm01} {\rm (\cite[(2.9)]{M:gendeg})} Let $\lam\vdash n$. Then $$
\gamma_{\tlam}=[\lam]!:=\prod_{i\geq 1}\lam_i!\,.
$$
\end{lemma}

\begin{theorem} {\rm (\cite[(2.14), (2.15)]{M:gendeg})} \begin{enumerate}
\item[(1)] $\bigl\{\tilde{f}_{\s\t}\bigm|\s,\t\in\Std(\lam),\lam\vdash n\bigr\}$ is a basis of matrix units in $\Q\Sym_n$.
\item[(2)] Let $\lam\vdash n$ and $\t\in\Std(\lam)$, then $F_{\t}=f_{\t\t}/\gamma_\t$ and $F_{\t}$ is a primitive idempotent in $\Q\Sym_n$ with $S^{\lam}\cong F_{\t}\Q\Sym_n$.
\item[(3)] For any $\lam\vdash n$ let $F_{\lam}:=\sum_{\t\in\Std(\lam)}F_{\t}$. Then $F_{\lam}$ is a primitive central idempotent in $\Q\Sym_n$.
\item[(4)] $\bigl\{F_{\lam}\bigm|\lam\vdash n\bigr\}$ is a complete set of primitive central idempotent in $\Q\Sym_n$ and $$
    1=\sum_{\lam\vdash n}F_{\lam}=\sum_{\text{$\t$ standard}}F_{\t}.
    $$
\end{enumerate}
\end{theorem}

\begin{lemma}\label{lm02} {\rm (\cite[Proposition 2.6]{M:gendeg})} Let $\lam\vdash n$ and $\s,\t\in\Std(\lam)$. Then $$
f_{\s\t}\equiv X_{\s\t}+\sum_{\substack{\u,\v\in\Std(\lam)\\
\u\rhd\s,\v\rhd\t}}a_{\u\v}X_{\u\v}\pmod{(\Q\Sym_n)^{\rhd\lam}},
$$
where $a_{\u,\v}\in\Q$ for each $\u,\v\in\Std(\lam)$.
\end{lemma}

\begin{definition} {\rm (\cite[Section 4]{DJ1})} Let $\lam\vdash n$. We define $$
z_{\lam}:=X_{\lam}w_{\lam}Y_{\lam'}.
$$
\end{definition}

\begin{lemma} \label{lm04} {\rm (\cite{DJ1}, \cite{Murphy:basis})} Let $\lam\vdash n$ and $w\in\Sym_n$. Then \begin{enumerate}
\item[(1)] $(\Z\Sym_n)^{\rhd\lam}Y_{\lam'}=0=Y_{\lam'}(\Z\Sym_n)^{\rhd\lam}$;
\item[(2)] If $w\neq w_{\lam}$ and $\ell(w)\leq\ell(w_{\lam})$, then $X_{\lam}wY_{\lam'}=0$ in $\Z\Sym_n$.
\end{enumerate}
\end{lemma}

\begin{proof} (1) follows from \cite[Lemma 4.12]{Murphy:basis}. It remains to prove (2). Assume that
$w\neq w_{\lam}$ and $\ell(w)\leq\ell(w_{\lam})$. If
$\ell(w)<\ell(w_{\lam})$, then by \cite[Corollary
4.13]{Murphy:basis} we see that $X_{\lam}wY_{\lam'}=0$. Now assume
$\ell(w)=\ell(w_{\lam})$. Then by \cite[Lemma 1.5]{DJ1} $\tlam
w\not\in\std(\lam)$ because $w\neq w_{\lam}$.

For any $\t\in\std(\lam)$ and integer $1\leq k<n$, it is well-known
that $$ X_{\lam}d(\t)s_k=\begin{cases} X_{\lam}d(\t), &\text{if
$k,k+1$ are in the same row of $\t$;}\\
X_{\lam}d(\t s_k), &\text{if $\t s_k\in\std(\lam)$,}
\end{cases}
$$
and if $k,k+1$ are in the same column of $\t$, then (by
\cite[Corollary 3.21]{M:Ulect}) $$ X_{\lam}d(\t)s_k\equiv
-X_{\lam}d(\t)+\sum_{\t\lhd\v\in\std(\lam)}r_{\v}X_{\lam}d(\v)\pmod{(\Z\Sym_n)^{\rhd\lam}},
$$
where $r_{\v}\in\Z$ for each $\v$. By \cite[Theorem 3.8]{M:Ulect},
$\t\lhd\v\in\std(\lam)$ only if $\ell(d(\t))>\ell(d(\v))$. As a
result (of the fact $\tlam w\not\in\std(\lam)$), we see that
$$
X_{\lam}w\equiv\sum_{\substack{\t\in\std(\lam)\\
\ell(d(\t))<\ell(w)=\ell(w_{\lam})}}b_{\t}X_{\lam}d(\t)\pmod{(\Z\Sym_n)^{\rhd\lam}},
$$
where $b_{\t}\in\Z$ for each $\t$. Now, applying the result (1)
(which we have just proved) and \cite[Corollary 4.13]{Murphy:basis}
again,
$$ X_{\lam}wY_{\lam'}=\sum_{\substack{\t\in\std(\lam)\\
\ell(d(\t))<\ell(w)=\ell(w_{\lam})}}b_{\t}X_{\lam}d(\t)Y_{\lam'}=0,
$$
as required. This completes the proof of (2).
\end{proof}

\begin{lemma}\label{lm03} {\rm (\cite[(3.13)]{M:gendeg})} Let $\lam\vdash n$. Then $$
z_{\lam}=\gamma_{\t^{\lam'}}f_{\tlam\t_{\lam}}.
$$
\end{lemma}

Note that in the right hand side of the above lemma the coefficient is $\gamma_{\t^{\lam'}}$ instead of $\gamma'_{\t^{\lam'}}$ because we have specialized $q$ to $1$.\smallskip

The next lemma is a key observation to the proof of Theorem \ref{mainthm}.

\begin{lemma} \label{keylem} Let $\lam=(n-k,k)\vdash n$, where $k\in\Z$ such that $0\leq k\leq n/2$. Then \begin{enumerate}
\item[(1)] $z_{\lam}=2^{k}f_{\tlam\t_{\lam}}\equiv 2^{k}X_{\tlam\t_{\lam}}\!\!\pmod{(\Z\Sym_n)^{\rhd\lam}}$.
\item[(2)] $Y_{\lam'}w_{\lam'}X_{\lam}w_{\lam}Y_{\lam'}=2^{k}Y_{\lam'}w_{\lam'}X_{\lam}w_{\lam}$.
\end{enumerate}
\end{lemma}

\begin{proof} By Lemma \ref{lm03} and Lemma \ref{lm01}, we get that $z_{\lam}=2^{k}f_{\tlam\t_{\lam}}$.
Applying Lemma \ref{lm02}, we get that $$
z_{\lam}=2^{k}f_{\tlam\t_{\lam}}\equiv 2^{k}X_{\tlam\t_{\lam}}\!\!\pmod{(\Q\Sym_n)^{\rhd\lam}}.
$$
Since $z_{\lam}, 2^{k}X_{\tlam\t_{\lam}}\in\Z\Sym_n$, we deduce that $$
z_{\lam}=2^{k}f_{\tlam\t_{\lam}}\equiv 2^{k}X_{\tlam\t_{\lam}}\!\!\pmod{(\Z\Sym_n)^{\rhd\lam}}.
$$
This proves (1). Note that $w_{\lam}^{-1}=w_{\lam'}$. By \cite[(4.1)]{DJ1} and applying the anti-automorphism $*$ we know that $Y_{\lam'}w_{\lam'}(\Z\Sym_n)^{\rhd\lam}=0$.  Now (2) follows from this equality and (1).
\end{proof}

Let $\lam,\mu$ be two compositions of $n$. A $\lam$-tableau of type
$\mu$ is a map $\bS: [\lam]\rightarrow\{1,2,\cdots,d\}$ such that
$\mu_i=\#\{\gamma\in[\lam]|\bS(\gamma)=i\}$ for $i\geq 1$. A
$\lam$-tableau $\bS$ of type $\mu$ is row semistandard if the
entries in each row of $\bS$ are non-decreasing from left to right;
$\bS$ is semistandard if (i) $\lam$ is a partition; and (ii) $\bS$
is row semistandard and the entries in each column of $\bS$ are
strictly increasing from top to bottom. If $\t\in\std(\lam)$ then we
define $\mu(\t)$ to be the $\lam$-tableau obtained from $\t$ by
replacing each entry $i$ in $\t$ by $r$ if $i$ appears in row $r$ of
$\t^{\mu}$. It is clear that $\mu(\t)$ is a row semistandard
$\lam$-tableau of type $\mu$. Recall our definition (see Section 2) of $X_{\lam}, Y_{\lam}$ for
each composition $\lam$. For each integer $i$ with $0\leq i\leq n$,
we set
$$ X_{i}:=X_{(i,n-i)},\quad  Y_{i}:=Y_{(i,n-i)}.
$$
For any $h\in K\Sym_n$, we use $\<h\>_0$ to denote the two-sided ideal of $K\Sym_n$ generated by $h$.
The next theorem is the first step in the direction towards the proof of Theorem \ref{conj}.

\begin{theorem} \label{mainthm} Let $K$ be a field of characteristic other than two. Let $a$ be an integer with $0\leq a\leq n/2$. Then we have that $$ \<X_{n-a}\>_0=\bigl(K\Sym_n\bigr)^{\unrhd(n-a,a)}.
$$
\end{theorem}

\begin{proof} By the cellular structure of $K\Sym_n$, we see that $\<X_{n-a}\>_0\subseteq\bigl(K\Sym_n\bigr)^{\unrhd(n-a,a)}$ and \begin{equation}\label{dfnna}
\dim\bigl(K\Sym_n\bigr)^{\unrhd(n-a,a)}=n_a:=\sum_{(n-a,a)\unlhd\lam\vdash n}\Bigl(\#\Std(\lam)\Bigr)^2.
\end{equation}
To prove the theorem, it suffices to find at least $n_a$ $K$-linear independent elements in the two-sided ideal $\<X_{n-a}\>_0$.

Let $\lam\vdash n$ be a partition such that $\lam\unrhd \mu:=(n-a,a)$. Then $\lam:=(n-k,k)$, where $k\in\Z$ with
$0\leq k\leq a\leq n/2$. Let $\bS_k$ be the following (unique) semistandard $\lam$-tableau of type $\mu$:
$$\bS_k:=\begin{matrix}
&\overbrace{1,1,\dots,1,}^{\text{$n-a$ copies}}&\overbrace{2,\dots,2,}^{\text{$a-k$ copies}}\\
&\underbrace{2,\dots,2}_{\text{$k$ copies}}\quad &
\end{matrix}
$$
We define $S_0(\lam):=\bigl\{\t\in\Std(\lam)\bigm|\mu(\t)=\bS_k\bigr\}$. By \cite[Section 7]{Murphy:basis}, \begin{equation}\label{eqa201}
\sum_{\t\in S_0(\lam)}X_{\lam}d(\t)\in (K\Sym_n)X_{n-a}\subseteq\<X_{n-a}\>_0.
\end{equation}
It is clear that $S_0(\lam)\neq\emptyset$. We divide the remaining proof into three steps:\medskip

{\it Step 1.} We claim that $Y_{\lam'}w_{\lam'}X_{\lam}\Bigl(\sum_{\t\in S_0(\lam)}d(\t)\Bigr)K\Sym_n=
Y_{\lam'}w_{\lam'}X_{\lam}K\Sym_n$.

Let $\t_0\in S_0(\lam)$ such that $\ell(d(\t_0))$ is maximal. Then
$\ell(d(\s))\leq\ell(d(\t_0))$ for any $\s\in S_0(\lam)$. Furthermore, by \cite[Lemma 1.5]{DJ1}, $$
\ell(w_{\lam})=\ell(d(\t_0))+\ell(d(\t_0)^{-1}w_{\lam}).
$$
We set $w:=d(\t_0)^{-1}w_{\lam}$. By Lemma \ref{lm04}, we deduce that $X_{\lam}d(\s)wY_{\lam'}=0$ for
any $\t_0\neq\s\in S_0(\lam)$. Now multiplying $wY_{\lam'}$ and applying Lemma \ref{lm04} and Lemma \ref{keylem},
we get that $$\begin{aligned}
&\quad\,Y_{\lam'}w_{\lam'}X_{\lam}\Bigl(\sum_{\t\in S_0(\lam)}d(\t)\Bigr)wY_{\lam'}\\
&=Y_{\lam'}w_{\lam'}X_{\lam}w_{\lam}Y_{\lam'}
&=2^{k}Y_{\lam'}w_{\lam'}X_{\lam}w_{\lam}.
\end{aligned}
$$
Since $\ch K\neq 2$, $2^{k}$ is invertible in $K$. The above equality implies that $$
Y_{\lam'}w_{\lam'}X_{\lam}\Bigl(\sum_{\t\in S_0(\lam)}d(\t)\Bigr)K\Sym_n=Y_{\lam'}w_{\lam'}X_{\lam}w_{\lam}K\Sym_n=
Y_{\lam'}w_{\lam'}X_{\lam}K\Sym_n,
$$
as required. This proves our claim. Furthermore, it is well-known that the elements in $\{Y_{\lam'}w_{\lam'}
X_{\lam}d(\t)|\t\in\Std(\lam)\}$ form a $K$-basis of $Y_{\lam'}w_{\lam'}X_{\lam}K\Sym_n$. In particular, we have
that $\dim Y_{\lam'}w_{\lam'}X_{\lam}K\Sym_n=\#\Std(\lam)$.
\medskip

{\it Step 2.} We define $$
N_{\lam}:=\frac{X_{\lam}\Bigl(\sum\limits_{\t\in S_0(\lam)}d(\t)\Bigr)K\Sym_n}{X_{\lam}
\Bigl(\sum\limits_{\t\in S_0(\lam)}d(\t)\Bigr)K\Sym_n\bigcap(K\Sym_n)^{\rhd\lam}}.
$$
We claim that $n_{\lam}:=\dim N_{\lam}\geq\#\Std(\lam)$.

In fact, by Lemma \ref{lm04}, the left multiplication by $Y_{\lam'}w_{\lam'}$ induces a surjective homomorphism from
$N_{\lam}$ onto $Y_{\lam'}w_{\lam'}X_{\lam}\Bigl(\sum_{\t\in S_0(\lam)}d(\t)\Bigr)K\Sym_n$. By the main result of Step
1, we know that $Y_{\lam'}w_{\lam'}X_{\lam}\Bigl(\sum_{\t\in S_0(\lam)}d(\t)\Bigr)K\Sym_n=
Y_{\lam'}w_{\lam'}X_{\lam}K\Sym_n$ and has dimension $\#\Std(\lam)$. It follows that
$n_{\lam}:=\dim N_{\lam}\geq\#\Std(\lam)$, as required. This proves our claim.

As a consequence, we can find $u_1,\cdots,u_{n_{\lam}}\in\Sym_n$ such that the natural image of the following elements
$$
X_{\lam}\Bigl(\sum_{\t\in S_0(\lam)}d(\t)\Bigr)u_1,\cdots,X_{\lam}\Bigl(\sum_{\t\in S_0(\lam)}d(\t)\Bigr)u_{n_{\lam}}
$$
in $N_{\lam}$ form a $K$-basis of $N_{\lam}$. For each $\s\in\Std(\lam)$ and each integer $1\leq j\leq n_{\lam}$,
we define $$
v_{\s,j}:=d(\s)^{-1}X_{\lam}\Bigl(\sum_{\t\in S_0(\lam)}d(\t)\Bigr)u_j. $$
By construction, it is clear that $v_{\s,j}\in\<X_{n-a}\>_0$.
\medskip

{\it Step 3.} We claim that the elements in the following set \begin{equation}\label{eq303}
\bigl\{v_{\s,j}\bigm|\s\in\Std(\lam), \lam\vdash n, 1\leq j\leq n_{\lam}\bigr\}
\end{equation}
are $K$-linearly independent.

In fact, assume that \begin{equation}\label{eqa404}
\sum_{\substack{\lam\vdash n, \s\in\Std(\lam)\\ 1\leq j\leq n_{\lam}}}c_{\s,j}v_{\s,j}=
\sum_{\substack{\lam\vdash n, \s\in\Std(\lam)\\ 1\leq j\leq n_{\lam}}}c_{\s,j}d(\s)^{-1}X_{\lam}\Bigl(\sum_{\t\in S_0(\lam)}d(\t)\Bigr)u_j=0, \end{equation}
where $c_{\s,j}\in K$ for each $\s, j$. Set $$
\Sigma_0:=\bigl\{(\s,j)\bigm|\s\in\Std(\lam), \lam\vdash n, 1\leq j\leq n_{\lam}, c_{\s,j}\neq 0\bigr\}.
$$
Suppose that $\Sigma_0\neq\emptyset$. We choose an $\s$ such that
$(\s,j)\in\Sigma_0$ for some $j$ and $\Shape(\s)=\lam$ is minimal
under the dominance order ``$\lhd$". By the property of the cellular
basis $\{X_{\s,\t}\}$ we know that $$ X_{\lam}\Bigl(\sum_{\t\in
S_0(\lam)}d(\t)\Bigr)u_j\equiv\sum_{\t\in\Std(\lam)}r_{j,\t}X_{\tlam,\t}\pmod{(K\Sym_n)^{\rhd\lam}},
$$
and hence $$
v_{\s,j}\equiv \sum_{\t\in\Std(\lam)}r_{j,\t}X_{\s,\t}\pmod{(K\Sym_n)^{\rhd\lam}}.
$$
By the main result in Step 2, we know that $X_{\lam}\Bigl(\sum_{\t\in S_0(\lam)}d(\t)\Bigr)u_j\not\in(K\Sym_n)^{\rhd\lam}$. It follows that at least one of those coefficients
$r_{j,\t}$ is nonzero. Combing this fact and the minimality of $\lam$ and (\ref{eqa404}) we can deduce that for each $\s\in\Std(\lam)$ and each $1\leq j\leq n_{\lam}$, $$
\sum_{1\leq j\leq n_{\lam}}c_{\s,j}d(\s)^{-1}X_{\lam}\Bigl(\sum_{\t\in S_0(\lam)}d(\t)\Bigr)u_j\in (K\Sym_n)^{\rhd\lam}.
$$
Now applying the main result of Step 2 again, we deduce that $c_{\s,j}=0$ for each $1\leq j\leq n_{\lam}$, a contradiction to the definition of $\Sigma_0$. This completes the proof of our claim.

As a consequence, we see that the two-sided ideal $\<X_{n-a}\>_0$ contains at least $$
\sum_{(n-a,a)\unlhd\lam\vdash n}n_{\lam}(\#\Std(\lam))\geq \sum_{(n-a,a)\unlhd\lam\vdash n}(\#\Std(\lam))^2=n_a
$$
$K$-linearly independent elements. This implies that we must have that $$
\dim \<X_{n-a}\>_0=n_a=\dim (K\Sym_n)^{\unrhd (n-a,a)}
$$
and hence $\<X_{n-a}\>_0=(K\Sym_n)^{\unrhd (n-a,a)}$. This completes the proof of the theorem.
\end{proof}

Our second step in the direction towards Theorem \ref{conj} is the proof of the following purely combinatorial identity.

\begin{theorem} \label{mainthm2} Let $K$ be a field of arbitrary characteristic. Let $S^{(m+1,m+1)}$ (respectively, $S^{(m+1-k,k)}$) be the Specht module of the symmetric group algebra $K\Sym_{2m+2}$ (respectively, $K\Sym_{m+1}$) corresponding to $(m+1,m+1)$ (respectively, $(m+1-k,k$), where $0\leq k\leq [(m+1)/2]$. Then we have that $$
\sum_{k=0}^{[(m+1)/2]}\bigl(\dim_K S^{(m+1-k,k)}\bigr)^2=\dim_K S^{(m+1,m+1)}.
$$
\end{theorem}

\begin{proof} We shall give a representation theoretic argument to prove the above combinatorial identity. Note that the dimension of each Specht module is independent of the field. We can assume without loss of generality that $K=\C$. \smallskip

We use $\bb_{m+1}(m)_{\C}$ to denote the specialized Brauer algebra over $\C$ with parameter $m$. In other words, we assume that $n=m+1$ for the moment. Then by \cite{DWH}, $\bb_{m+1}(m)_{\C}$ is semisimple. We have a natural surjective homomorphism $\varphi_{\C}: \bb_{m+1}(m)_{\C}\twoheadrightarrow\End_{\C O(V_{\C})}\bigl(V_{\C}^{\otimes m+1}\bigr)$,
where $V_{\C}$ is an orthogonal $\C$-vector space with dimension $m$. Note that $\C\Sym_{2m+2}$ is semisimple and each Specht module $\C\Sym_{2m+2}$ over is irreducible. By \cite{DH} we see that \begin{equation}\label{fstdim} \dim\Ker\varphi_{\C}=\dim_{\C}S_{\C}^{(m+1,m+1)}, \end{equation}
where we use $S_{\C}^{(m+1,m+1)}$ to denote the simple (Specht) module of $\C\Sym_{2m+2}$ corresponding to the partition $(m+1,m+1)$. As before, we use $\bb_{m+1}^{(1)}$ to denote the two-sided ideal of $\bb_{m+1}(m)_{\C}$ generated by $e_1$. Equivalently, $\bb_{m+1}^{(1)}$ is the subspace of $\bb_{m+1}(m)_{\C}$ spanned by those Brauer $(m+1)$-diagrams which contain horizontal edges. We claim that $\Ker\varphi_{\C}\cap\bb_{m+1}^{(1)}=\{0\}$.

In fact, by \cite[Theorem 4.3]{LZ2}, we know that $\Ker\varphi_{\C}$ is equal to the two-sided ideal $\<E_{[(m+1)/2]}\>$. It suffices to show that $\<E_{[(m+1)/2]}\>\cap\bb_{m+1}^{(1)}=\{0\}$. Since $\bb_{m+1}^{(1)}$ is a two-sided ideal of the semisimple $\C$-algebra $\bb_{m+1}(m)_{\C}$, we have that $\bb_{m+1}(m)_{\C}=\bb_{m+1}^{(1)}\oplus I_0$ for some two-sided ideal $I_0$ of $\bb_{m+1}(m)_{\C}$. In particular, $xI_0=0=I_0x$ for any $x\in\bb_{m+1}^{(1)}$. We define $$
J_0:=\bigl\{y\in\bb_{m+1}(m)_{\C}\bigm|\text{$yx=0=xy$ for any $x\in\bb_{m+1}^{(1)}$}\bigr\}.
$$
Then $J_0$ is a two-sided ideal of $\bb_{m+1}(m)_{\C}$ and $I_0\subseteq J_0$. Thus $\bb_{m+1}(m)_{\C}=J_0+\bb_{m+1}^{(1)}$.
Since $J_0x=0$ for any $x\in\bb_{m+1}^{(1)}$, it follows that every right simple submodule of $J_0$ is a simple module over $\bb_{m+1}(m)_{\C}/\bb_{m+1}^{(1)}\cong\C\Sym_{m+1}$. Using the Wedderburn theorem for semisimple algebras we get that $$
J_0=\bigoplus_{\lam\vdash m+1}\bigl(S_{\C}^{\lam}\bigr)^{\oplus a_{\lam}},
$$
where $a_{\lam}\in\N$ such that $0\leq a_{\lam}\leq\dim S_{\C}^{\lam}$ for each $\lam\vdash m+1$. As a result, we deduce that $$
\dim J_0\leq\sum_{\lam\vdash m+1}(\dim S_{\C}^{\lam})^2=\dim_{\C}\C\Sym_{m+1}=\dim_{\C}\bb_{m+1}(m)_{\C}-\dim_{\C}\bb_{m+1}^{(1)},
$$
because $\bb_{m+1}(m)_{\C}/\bb_{m+1}^{(1)}\cong\C\Sym_{m+1}$. This in turn forces $$
J_0\bigoplus\bb_{m+1}^{(1)}=J_0+\bb_{m+1}^{(1)}=\bb_{m+1}(m)_{\C}.
$$
In particular, $J_0\cap\bb_{m+1}^{(1)}=\{0\}$. On the other hand, by \cite[Corollary 5.13]{LZ2} we know that $\<E_{[(m+1)/2]}\>
\subseteq J_0$, from which our claim follows at once.\medskip

We write $a_0:=m+1-[(m+1)/2]$. Then $a_0\geq m+1-a_0$ and hence $(a_0,m+1-a_0)$ is a partition of $m+1$. Note that (by Lemma \ref{lm00})
$$\begin{aligned} & \langle E_{[(m+1)/2]}\rangle=\langle E_{[(m+1)/2],a_0}\rangle=\langle E_{a_0,[(m+1)/2]}\rangle=\langle E_{a_0}\rangle,\\
& E_{a_0}=E_{a_0,m+1-a_0}\equiv Y_{a_0,m+1-a_0}\equiv Y_{a_0}\pmod{\bb_{m+1}^{(1)}}. \end{aligned} $$

By Theorem \ref{mainthm}, we know that
$\<Y_{a_0}\>_0=(\C\Sym_{m+1})^{\unrhd(a_0,m+1-a_0)}$. Combining these fact together with the equality $$
\langle E_{[(m+1)/2]}\rangle\cap\bb_{m+1}^{(1)}=\Ker\varphi_{\C}\cap\bb_{m+1}^{(1)}=\{0\}$$ which we have just proved, we can deduce that $$\begin{aligned}
\dim\Ker\varphi_{\C}&=\dim\langle E_{[(m+1)/2]}\rangle=\dim(\C\Sym_{m+1})^{\unrhd(a_0,m+1-a_0)}\\
&=\sum_{k=0}^{[(m+1)/2]}\bigl(\dim_{\C} S_{\C}^{(m+1-k,k)}\bigr)^2.
\end{aligned}$$
Finally, comparing the above equality with (\ref{fstdim}), we complete the proof of the theorem.
\end{proof}

By the well-known hook formula (cf. \cite{M:Ulect}) for the dimension of Specht modules, we get that $$
\dim S^{(m+1)}=1,\quad \dim S^{(m,1)}=m,
$$
and for each $2\leq k\leq [(m+1)/2]$, $$\begin{aligned}
\dim S^{(m+1-k,k)}&=\frac{(m+1)!}{\prod_{(i,j)\in[(m+1-k,k)]}h_{i,j}^{(m+1-k,k)}}\\
&=\frac{(m+1)m\cdots (m-k+4)(m-k+3)(m-2k+2)}{k!},
\end{aligned} $$
where $h_{i,j}^{\lam}:=\lam_i+\lam'_j-i-j+1$ is the $(i,j)$-hook length of the partition $\lam$. Similarly, $$\begin{aligned}
\dim S^{(m+1,m+1)}&=\frac{(2m+2)!}{\prod_{(i,j)\in[(m+1,m+1)]}h_{i,j}^{(m+1,m+1)}}\\
&=\frac{(2m+2)(2m+1)\cdots (m+3)}{(m+1)!}.
\end{aligned}$$

As a result, we have the following corollary.

\begin{corollary} \label{maincor} We have the following identity: $$\begin{aligned}
&\quad\,1+m^2+\sum_{k=2}^{[(m+1)/2]}\biggl(\frac{(m+1)m\cdots (m-k+3)(m+2-2k)}{k!}\biggr)^2\\
&=\frac{(2m+2)(2m+1)\cdots (m+3)}{(m+1)!}.
\end{aligned} $$
\end{corollary}

\begin{proof} This follows directly from Theorem \ref{mainthm2} and the hook length formulae for the dimensions of Specht modules over symmetric groups.
\end{proof}

Now we are in the final step to prove Theorem \ref{conj}.

\begin{theorem} \label{mainthm3} Let $K$ be an arbitrary field of characteristic other than two. Then the annihilator of $V^{\otimes n}$ in $\bb_n(m)$ is the two-sided ideal generated by $E_{[(m+1)/2]}$.
\end{theorem}

\begin{proof} By Theorem \ref{xx3.6}, it suffices to show that for each integer $i$ with $0\leq i<[(m+1)/2]$, $E_i$
lies in the two-sided ideal of $\bb_n(m)$ generated by
$E_{[(m+1)/2]}$. For simplicity, we write $a:=[(m+1)/2]$. By Lemma \ref{lm00}, we know that $\<E_{j}\>=\<E_{m+1-j}\>$ for each $0\leq j\leq m+1$.
Therefore, it suffices to show that for each integer $i$ with $0\leq i<a$, $E_{(m+1-i,i)}$ lies in the
two-sided ideal of $\bb_n(m)$ generated by $E_{m+1-a}:=E_{(m+1-a,a)}$. Note that for each
integer $0\leq i\leq m+1$, $E_i\in\bb_{m+1}(m)\subseteq\bb_n(m)$. Without loss of generality, we can assume that
$n=m+1$ henceforth. \smallskip

We consider the natural homomorphism $\varphi_K: \bb_{m+1}(m)\rightarrow\End_K(V^{\otimes m+1})$.  By Corollary \ref{maincor}, we have that $$\begin{aligned}
(m+1)_a&:=1+m^2+\sum_{k=2}^{[(m+1)/2]}\biggl(\frac{(m+1)m\cdots (m-k+3)(m+2-2k)}{k!}\biggr)^2\\
&=\frac{(2m+2)(2m+1)\cdots (m+3)}{(m+1)!}.
\end{aligned}$$
By \cite[Lemma 7.1]{DH}, Theorem \ref{mainthm2} and Corollary \ref{maincor}, we have that $$
\dim_K\Ann_{\bb_{m+1}(m)}(V^{\otimes m+1})=\dim_K\Ker\varphi_{K}=(m+1)_a.
$$
On the other hand, since $E_{m+1-a}\equiv Y_{m+1-a}\pmod{\bb_n^{(1)}}$, we deduce (from Theorem \ref{mainthm2}, Corollary \ref{maincor} and the $Y$-Murphy basis for $K\Sym_{m+1}$) that $\<E_{m+1-a}\>$ contains at least $(m+1)_a$ $K$-linearly independent elements. Since $\<E_{m+1-a}\>\subseteq\Ker\varphi_{K}$, it follows that
$\<E_{m+1-a}\>=\Ker\varphi_{K}$ and hence for each integer $0\leq i\leq [(m+1)/2]$, $
E_i\in\<E_{m+1-a}\>=\<E_a\>=\<E_{[(m+1)/2]}\>$,
as required. This completes the proof of the Theorem.
\end{proof}

The following corollary give a new integral basis for the annihilator of $V^{\otimes m+1}$ in $\bb_{m+1}(m)$.

\begin{corollary} \label{maincor2} Let $K$ be a field of characteristic other than two. Then the elements in the following set $$
\bigl\{d(\s)^{-1}E_{[(m+1)/2]}d(\t)\bigm|\s,\t\in\std(m+1-k,k), 0\leq k\leq [(m+1)/2]\bigr\}
$$
form a $K$-basis of $\Ann_{\bb_{m+1}(m)}(V^{\otimes m+1})$
\end{corollary}

\begin{proof} This follows directly from the proof of Theorem \ref{mainthm3}.
\end{proof}

\medskip

\centerline{\bf Acknowledgment}\medskip

The first author is supported by the National Natural Science
Foundation of China. The second author is supported by a research
foundation of Huaqiao University (Grant No. 10BS323).

\end{document}